\documentclass[11pt, a4paper,reqno]{amsart}

\usepackage{amssymb, xcolor}

\numberwithin{equation}{section}

\newtheorem{thm}{Theorem}[section]
\newtheorem{lemma}[thm]{Lemma}
\newtheorem{prop}[thm]{Proposition}
\newtheorem{coro}[thm]{Corollary}
\theoremstyle{definition}
\newtheorem{definition}[thm]{Definition}
\newtheorem{remark}[thm]{Remark}

\newcommand{\Z}{\mathbb{Z}}

\newcommand{\C}{\mathbb{C}}
\newcommand{\ot}{\otimes}

\title[Soergel bimodules for complex reflection groups of rank one]{A Soergel-like category for complex reflection groups of rank one}

\author{Thomas Gobet}
\address{School of Mathematics and Statistics F07, University of Sydney NSW 2006, Australia.}
\email{thomas.gobet@sydney.edu.au}
\author{Anne-Laure Thiel}
\address{Universit\"at Stuttgart, Fachbereich Mathematik, Institut f\"ur Geometrie und Topologie, Pfaffenwaldring 57, 70569 Stuttgart Cedex, Germany.}
\email{anne-laure.thiel@mathematik.uni-stuttgart.de }

\begin{document}

\maketitle

\begin{abstract}
We introduce analogues of Soergel bimodules for complex reflection groups of rank one. We give an explicit parametrization of the indecomposable objects of the resulting category and give a presentation of its split Grothendieck ring by generators and relations. This ring turns out to be an extension of the Hecke algebra of the reflection group $W$ and a free module of rank $|W| (|W|-1)+1$ over the base ring. We also show that it is a generically semisimple algebra if defined over the complex numbers.  
\end{abstract}

\tableofcontents

\date{December 4, 2018}

\section{Introduction}

The aim of this paper is to introduce analogues of Soergel bimodules for complex reflection groups of rank one. 

Finite complex reflection groups are generalizations of finite real reflection groups, also known as finite Coxeter groups. Every (not necessarily finite) Coxeter group $W$ has a faithful linear representation as a group generated by reflections on a real vector space, preserving a symmetric bilinear form.

Given a Coxeter group $W$ with set of simple generators $S$, Soergel \cite{S1}, \cite{S} gave a way to categorify the Iwahori-Hecke algebra of $W$ using a so-called \emph{reflection faithful} representation $V$ of $W$, which is a finite-dimensional faithful reflection representation of $W$ satisfying some properties. The Iwahori-Hecke algebra is then realized as the (split) Grothendieck ring of a category of graded bimodules over the ring $R=\mathcal{O}(V)$ of regular functions on $V$. This allows one to categorify the Kazhdan-Lusztig polynomials~\cite{KL}, which are ubiquitous in the representation theory of Lie theoretic objects and deeply connected to the geometry of Schubert varieties~\cite{KL, KL2}. Soergel bimodules were used to solve important conjectures (such as the non-negativity of the coefficients of these polynomials \cite{EW}); in this framework, the Kazhdan-Lusztig polynomials are interpreted as graded multiplicities in a canonical filtration of the indecomposable Soergel bimodules. 

Soergel bimodules are also of interest outside the Lie theoretic world, as the Artin-Tits group attached to $W$ has a categorical action on the bounded homotopy category of Soergel bimodules, as shown by Rouquier~\cite{Rouq}, \cite{Rouqprep}. This construction can be made for an arbitrary (finitely generated) Coxeter group, and the action had been proven to be faithful if $W$ is finite~\cite{KS, Jensen} (faithfulness is conjectured in general).   

In the case of a finite Weyl group, Soergel bimodules describe the equivariant intersection cohomology of a Schubert variety~\cite{Sperv}. Hence they can be thought of as some kind of extension to arbitrary Coxeter groups of the intersection cohomology of (in general non-existing) flag varieties. 

There have been attempts to generalize several objects associated to finite Weyl groups or Coxeter groups to complex reflection groups. One can cite for instance the 'Spetses' program~\cite{Spetses}, which provides a sort of generalization of unipotent characters of reductive groups to non-existing reductive groups attached to complex reflection groups. In this framework, some categorification results were obtained for cyclic groups in \cite{BoRo} (building up on \cite{EGST1, EGST2}) and later extended in \cite{La1, La2}. Note that the ring considered in \cite[Theorem 5.5]{BoRo} is related to the ring $A_W$ studied in the present paper, as observed in Remark \ref{remBoRo}. Furthermore, the Artin-Tits groups, which as mentioned above are categorified using complexes of Soergel bimodules, also admit nice generalizations to the complex case~\cite{BMR}. For these reasons, it is natural to try to extend Soergel's construction to complex reflection groups. 

This paper proposes the construction of the analogue of a category of Soergel bimodules for finite complex reflection groups of rank one. In the Coxeter group case, the category of Soergel bimodules is a graded category monoidally generated by a family of bimodules $\{B_s\}$ over the graded ring $R=\mathcal{O}(V)$, indexed by the simple reflections $s\in S$. Each bimodule $B_s$ admits both an algebraic definition as the tensor product $R\otimes_{R^s}R[1]$ (here $R^s\subseteq R$ is the graded subring of $s$-invariant functions and $[1]$ denotes a grading shift) and an equivalent definition as a ring of regular functions on a graph. In the complex case, both definitions can be given, but they produce non-isomorphic bimodules if the reflection does not have order $2$. In the rank one case, the algebraic definition leads to a category with only two indecomposable objects, hence its Grothendieck ring is an algebra which is a free module of rank two over the base ring (see Remark~\ref{rmq:FS} below). This does not give a very interesting algebra; in particular, in the case of complex reflection groups, one can still define a Hecke algebra~\cite{BMR}, and while it is not clear, to many extents, that this algebra is the "right" generalization of the Iwahori-Hecke algebra of a Coxeter group, it is natural to search for a Soergel-like category producing a Grothendieck ring which is connected to the Hecke algebras of~\cite{BMR}. For this reason, we choose to work with the geometric definition of the elementary Soergel bimodules as rings of regular functions on reversed graphs (see Section~\ref{sec:bim} below). With this definition, we obtain much more interesting categories, which we study in detail in this paper.

Writing $s$ for a reflection in $\mathbb{C}$ of order $d>2$, $W$ for the cyclic group generated by $s$, and $\mathcal{B}_W$ for the category obtained as the idempotent completion of the additive graded monoidal category generated by the analogue of the Soergel bimodule $B_s$ attached to $s$ (see Section~\ref{sec:bim} for precise definitions), we obtain the following description of the split Grothendieck ring of $\mathcal{B}_W$ (see Theorem~\ref{classif} and Proposition~\ref{pres1})

\begin{thm}[Structure of the split Grothendieck ring of $\mathcal{B}_W$]\label{thm:main}
Let $W, d, \mathcal{B}_W$ be as above. 
\begin{enumerate}
\item The indecomposable objects in $\mathcal{B}_W$ are, up to isomorphism and grading shifts, indexed by the (nonempty) cyclically connected subsets of $W$. In particular, the split Grothendieck ring $A_W:=\langle \mathcal{B}_W \rangle$ is a $\mathbb{Z}[v, v^{-1}]$--algebra which is a free $\mathbb{Z}[v, v^{-1}]$--module of rank $d (d-1)+1$. 
\item The algebra $A_W$ has a presentation with generators $s$, $C_1, \cdots, C_{d-1}$ and relations 
$$\begin{cases}
s^d=1,\\
C_i C_j=C_j C_i~\forall i,j\text{~and~} s C_i= C_i s~\forall i,\\
C_1 C_i=C_{i+1} + s C_{i-1}~\forall i=1,\dots, d-2,\\
C_1 C_{d-1}= \left(v+v^{-1}\right) C_{d-1},\\
s C_{d-1}= C_{d-1},
\end{cases}$$
with the convention that $C_0 := 1$. In particular $A_W$ is commutative, and has a subalgebra isomorphic to the group algebra of $W$. 
\end{enumerate} 
\end{thm}

Note that the third relation in the above presentation implies that the Grothendieck ring can be generated by the two elements $s$ and $C_1$, while the category $\mathcal{B}_W$ is generated by a single object (whose isomorphism class is $C_1$). The reason for this is that we take the idempotent completion of the category generated by $B_s$. In this situation, whenever $d>2$, there are indecomposable bimodules appearing, which turn out to be invertible. In the Coxeter case, such bimodules can also be defined but (except for the monoidal identity) they do not belong to Soergel's category $\mathcal{B}$; but they are the constituents of the canonical filtrations of the indecomposable objects which categorify the Kazhdan-Lusztig polynomials~\cite{S, EW}. It is not known what the Grothendieck ring of a category generated by Soergel bimodules and these invertible bimodules is in the case of a Coxeter group, except in type $A_1$ (see Remark \ref{remA1}) and in type $A_2$ where it was described by the authors in \cite{GT} and gives rise to an algebra which is a free module of rank $25$ over the base ring. This algebra is also an extension of the Iwahori-Hecke algebra (and a quotient of the affine Hecke algebra of type $\widetilde{A}_2$). Hence, describing such a category is an open problem even for finite Coxeter groups (and even for dihedral groups). On the other hand, the category $\mathcal{B}_W$ could also be considered as the exact analogue of Soergel's category. 

It is natural to study the similarities between the algebra $A_W$ and the Hecke algebra $H_W$ associated to $W$. We write $A_W^{\mathbb{C}}$ for the algebra defined by the same presentation as the one given in Theorem~\ref{thm:main} above, but defined over the complex numbers, with $v\in \mathbb{C}^\times$. Recall that Hecke algebras associated with finite reflection groups are generically semisimple. We show (see Theorem~\ref{thm:semi}) 

\begin{thm}[Semisimplicity]
The algebra $A_W^{{\mathbb{C}}}$ is generically semisimple. More precisely, if $v+v^{-1}\neq 2\cos\bigl(\frac{k \pi}{d}\bigr)$ for all $k=1, \dots, d-1$, then $A_W^{{\mathbb{C}}}$ is semisimple. 
\end{thm}

Note that $\left\{2\cos\left(\frac{k \pi}{d}\right)~|~k=1, \dots, d-1 \right\}$ is the set of roots of a Chebyshev polynomial of the second kind.\\ 
~\\
{\bf{Acknowledgments}}. We thank Pierre-Emmanuel Chaput, Eirini Chavli, Anthony Henderson, Ivan Marin and Ulrich Thiel for useful discussions. The first author was funded by an ARC grant (grant number DP170101579).   

\section{Graded bimodules over polynomial rings and regular functions on graphs}\label{sec:bim}

Let $V$ be a finite-dimensional vector space over a field $k$ of characteristic zero. Let $W$ be a (finitely generated) reflection group, that is, a group generated by finitely many elements $s\in \mathrm{GL}(V)$ such that $s$ has finite order and $H_s:=\ker(s - \mathrm{id}_V)$ is a hyperplane. Note that since $s$ has finite order and $k$ has characteristic zero, the hyperplane $H_s$ has a one-dimensional complement $L$ which is $s$-stable, hence on which $s$ acts by a scalar. 

Let $R=\mathcal{O}(V)$ denote the $k$-algebra of regular functions on $V$. Since $k$ is infinite, we have $\mathcal{O}(V)\cong S(V^*)$, hence $R$ is graded, and inherits an action of $W$ from $V$. Note that $W$ acts degreewise. We adopt the convention that $\mathrm{deg}(V^*)=2$. All the $R$--bimodules which we will consider are graded, with the same operation of the field on both sides. Given an $R$--bimodule, we may talk about the \textit{left} (resp. \textit{right}) action of $R$ to denote the action of $R\otimes_{k} 1$ (resp. the action of $1\otimes_{k} R$), even if each action is both a left and a right action (since $R$ is commutative). All the bimodules which we shall consider are finitely generated as left and right $R$--modules. Note that the category of graded $R$--bimodules satisfies the Krull-Schmidt property, that is, every graded $R$--bimodule is a direct sum of indecomposable bimodules, and the indecomposable summands are unique up to isomorphism and permutation (see~\cite[Remark 1.3]{S}). 

If $k=\mathbb{R}$ and $W$ is finite, then $W$ is a finite Coxeter group. Coxeter groups appear in many different contexts and include the Weyl groups of all semisimple algebraic groups. Conversely, every (not necessarily finite) Coxeter group admits a canonical faithful linear representation as a reflection group, the Tits representation (see~\cite{Bou} or \cite{Humph}). Let $S$ denote the set of generators of $W$ associated to the walls of a chamber and $T=\bigcup_{w\in W} wSw^{-1}$ its set of reflections. Soergel has shown that the split Grothendieck ring of the additive graded monoidal Karoubian category $\mathcal{B}$ generated by the $R$--bimodules $R\otimes_{R^s} R$, $s\in S$, and stable by grading shifts, is isomorphic to the Iwahori-Hecke algebra $H_W$ of $W$ (see~\cite{S1}, \cite{S}). Here $R^s\subseteq R$ denotes the graded subring of $s$-invariant functions. The classes of the (unshifted) indecomposable objects in the split Grothendieck ring coincide with the elements of the Kazhdan-Lusztig basis of $H_W$ (this is Soergel's conjecture, proven in~\cite{EW}) and the bimodule $R\otimes_{R^s}R[1]$ corresponds to the Kazhdan-Lusztig generator $C'_s$. Note that Soergel requires $V$ to be a \textit{reflection faithful} representation of $W$ (\cite[Definition 1.5]{S}), that is, the reflections in $W$ must act by geometric reflections, and distinct reflections have distinct reflecting hyperplanes. If $W$ is finite, one can simply take the Tits representation of $W$. In the case of complex reflection groups, the notion of a reflection faithful representation does not really make sense since every non-trivial power of a reflection is still a reflection with the same hyperplane. We will simply consider a space $V$ on which $W$ acts irreducibly as a reflection group.      

For $x\in W$, let $$\mathrm{Gr}(x)=\{ (xv, v)~|~v\in V\}\subseteq V\times V$$ be the (reversed) graph of $x$. It is a Zariski-closed subset of $V\times V$. Given a subset $A=\{x_1, \dots, x_k\}$ of $W$, we denote by $\mathcal{O}(A)$ or simply $\mathcal{O}(x_1, x_2, \dots, x_k)$ the ring of regular functions on $\bigcup_{i=1}^k \mathrm{Gr}(x_i)$. Note that the two projections on $V$ define a structure of (graded) $R$--bimodule on $\mathcal{O}(A)$. If $W$ is a Coxeter group and $t$ is any reflection in $W$, then 
$$R\otimes_{R^t} R\cong \mathcal{O}(e,t)$$
(see~\cite[Remark 4.3]{S}). 

In this paper, we are interested in the case where $k=\mathbb{C}$, $V$ is one-dimensional and $W$ is finite. The structure of $W$ is very elementary since, in this case, it is cyclic and every element in $W\backslash\{1\}$ is a reflection. All the reflections share the same hyperplane $H_s$, which is reduced to $\{0\}$. Denote by $s$ a generator of $W$ and by $\zeta$ its only eigenvalue distinct from $1$. In fact, since $\mathrm{GL}(\mathbb{C})=\mathbb{C}^{\times}$, we could write $s=\zeta$ but we prefer to distinguish the eigenvalue and the linear transformation, since several statements make sense for other reflection groups.   

Write $W$ as $\bigl\{1, s, s^2, \dots, s^{d-1}\bigr\}$. 

\begin{definition} A subset $A\subseteq W$ is \emph{cyclically connected} if there are $0\leq i \leq d-1$ and $0\leq j \leq d-1$ such that $A=\left\{s^i, s^{i+1}, \dots, s^{i+j}\right\}$.
\end{definition}

For example, if $s$ has order $4$, there are $13$ cyclically connected subsets, namely 
\begin{multline*} 
\{e\}, \{s\}, \big\{s^2\bigr\}, \bigl\{s^3\bigr\}, \{e,s\}, \bigl\{s,s^2\bigr\}, \bigl\{s^2, s^3\bigr\}, \bigl\{s^3,e\bigr\},\\
\bigl\{e, s, s^2\bigr\}, \bigl\{s, s^2, s^3\bigr\}, \bigl\{s^2, s^3, e\bigr\}, \bigl\{s^3, e, s\bigr\}, \bigl\{e, s, s^2, s^3\bigr\} = W.
\end{multline*}
In particular, according to that definition, the empty set is not a cyclically connected subset of $W$, but the whole group $W$ is a cyclically connected subset. Hence there are $d (d-1)+1=(d-1)^2 + d$ such sets.

Note that $R\cong\mathbb{C}[X]$, where $X\in V^*$ is an equation of the hyperplane $\{H_s=0\}$. We have $s\cdot X= \zeta^{-1} X$. Given $0\leq i \leq d-1$ and $0\leq j \leq d-1$ as in the above definition, we write $\mathcal{O}\bigl(s^{[i, i+j]}\bigr)$ for the bimodule $\mathcal{O}(A)$ where $A=\bigl\{s^i, s^{i+1}, \dots, s^{i+j}\bigr\}$. If $i=0$ then we simply write $\mathcal{O}\bigl(s^{\leq j}\bigr)$. For example, $\mathcal{O}(e,s)=\mathcal{O}\bigl(s^{\leq 1}\bigr)$. 

\begin{remark}\label{indec}
Note that, for every (not necessarily cyclically connected) $A \subseteq W$, the $R$--bimodule $\mathcal{O}(A)$ is indecomposable. Indeed $\bigcup_{x\in A} \mathrm{Gr}(x)$ is a closed subscheme of $V\times V$, inducing a surjective map $\mathcal{O}(V\times V)\cong R\otimes_{\mathbb{C}} R\twoheadrightarrow \mathcal{O}(A)$ of algebras, compatible with the bimodule structure induced by the projections. It follows that $\mathcal{O}(A)$ is generated, as a graded $R$-bimodule, by any non-zero element in its degree zero component, hence that it is indecomposable, since this component is one-dimensional.

Moreover, observe that, for every $A \subseteq W$, the left and right actions of $R^s$ on  $\mathcal{O}(A)$ are the same. Indeed, for $b\in R^s$ and $a \in \mathcal{O}(A)$, we have
$$ (a \cdot b) (u,v) = a(u,v)b(v)=a(u,v)b(u)=(b \cdot a) (u,v)$$
for all $(u,v) \in \bigcup_{x\in A} \mathrm{Gr}(x)$, where the middle equality follows from the fact that $b\in R^s$ and $u=s^i v$ for some $i$.
\end{remark}

We shall study the (additive, graded, monoidal, Karoubian) category $\mathcal{B}_W$ generated by $\mathcal{O}\bigl(s^{\leq 1}\bigr)$ and stable by grading shifts. To this end, we need several technical results to understand how to decompose tensor products of rings of regular functions over cyclically connected subsets of $W$ (viewed as graded $R$--bimodules). 

Note that as a graded $R$--bimodule, $\mathcal{O}\bigl(s^i\bigr)$ is isomorphic to $R_{s^i}$, that is, $R$ with the right operation of $R=\mathbb{C}[X]$ twisted by $s^i$ (for $r\in R_{s^i}$, we have $r\cdot X= s^i(X)r= \zeta^{-i}Xr$, while $X\cdot r=Xr$). Indeed the embedding $\iota : V \hookrightarrow V\times V, v\mapsto\bigl(v, s^{-i}v\bigr)$ induces an isomorphism $\mathcal{O}\bigl(s^i\bigr) \xrightarrow{\sim} R_{s^i}, a \mapsto a \circ \iota$ of graded bimodules. It immediately follows that
\begin{equation}\label{isostwists}
\mathcal{O}\bigl(s^i\bigr)\otimes_R \mathcal{O}\bigl(s^j\bigr)\cong \mathcal{O}\bigl(s^{i+j}\bigr)\cong \mathcal{O}\bigl(s^j\bigr)\otimes_R \mathcal{O}\bigl(s^i\bigr),
\end{equation}
in particular, these rings give a categorification of $W$. In most of the calculations, we will identify $\mathcal{O}\bigl(s^i\bigr)$ with the ring $R=\mathcal{O}(V)$ with right operation twisted by $s^i$ and hence implicitly identify $a\bigl(u, s^{-i}u \bigr)$ with $a(u)$. 

\begin{lemma}\label{lem:isos}
There are isomorphisms of graded $R$--bimodules 
$$\mathcal{O}\bigl(s^i\bigr)\otimes_R \mathcal{O}\bigl(s^{\leq j}\bigr)\cong \mathcal{O}\bigl(s^{[i, i+j]}\bigl)\cong \mathcal{O}\bigl(s^{\leq j}\bigr) \otimes_R \mathcal{O}\bigl(s^i\bigr).$$ 
\end{lemma}

\begin{proof}
The isomorphism $\mathcal{O}\bigl(s^i\bigr)\otimes_R \mathcal{O}\bigl(s^{\leq j}\bigr)\cong \mathcal{O}\bigr(s^{[i, i+j]}\bigl)$ is given by 
$$a\otimes b\mapsto \bigl\{(u,v)\mapsto a(u) b\bigl(s^{-i}u,v\bigr)\bigr\}.$$
The isomorphism $\mathcal{O}\bigl(s^{\leq j}\bigr)\otimes_R \mathcal{O}\bigl(s^i\bigr)\cong \mathcal{O}\bigl(s^{[i, i+j]}\bigr)$ is given by 
$$a\otimes b\mapsto \bigl\{ (u,v)\mapsto a\bigl(u, s^i v \bigr) b\bigl(s^i v\bigr)\bigr\}.$$  
\end{proof}
  
In particular, since $\mathcal{O}\bigl(s^{\leq d-1}\bigr) = \mathcal{O}(W)$, these isomorphisms of graded $R$--bimodules specialize, for all $0\leq i \leq d-1$, to
\begin{equation}\label{isosFS}
\mathcal{O}\bigl(s^i\bigr)\otimes_R \mathcal{O}(W)\cong \mathcal{O}(W)\cong \mathcal{O}(W) \otimes_R \mathcal{O}\bigl(s^i\bigr).
\end{equation}

\begin{coro}\label{cor:isos}
There are isomorphisms of graded $R$--bimodules 
$$\mathcal{O}\bigl(s^i\bigr)\otimes_R \mathcal{O}(A)\cong \mathcal{O}(A) \otimes_R \mathcal{O}\bigl(s^i\bigr),$$
for all $i=1, \dots, d-1$ and $A$ cyclically connected.
\end{coro}
 
Let $0\leq i \leq d-1$. Note that 
$$I_i:=I\bigl(\mathrm{Gr}(e)\cup\mathrm{Gr}(s)\cup \dots\cup \mathrm{Gr}\bigl(s^i\bigr)\bigr)$$
 is a homogeneous ideal (equivalently, it is a graded $R$--bimodule), since $I(\mathrm{Gr}(x))$ is homogeneous for every $x\in W$. Let 
$$P_i(X,Y):=(X-Y)(X-\zeta Y)\cdots (X-\zeta^i Y)\in\mathbb{C}[X,Y].$$

Note that
\begin{equation}\label{Pd-1}
P_{d-1}(X,Y) = X^d - Y^d.
\end{equation}

\begin{lemma}\label{generateur}
Let $0\leq i \leq d-1$. Then $I_i$ is generated (as homogeneous ideal, equivalently as graded $R$--bimodule) by $P_i$.
\end{lemma}

\begin{proof}
We argue by induction on $i$. If $i=0$, then $I_0=I(\mathrm{Gr}(e))$ is just the ideal of the diagonal $\Delta$ in $V\times V$ which is generated by $X-Y$. Hence, we can assume that $i\geq 1$ and that the result holds for $i-1$. Let $P(X,Y)\in I_i$ be homogeneous. By induction, there exists $Q\in\mathbb{C}[X,Y]$ such that 
$$P(X,Y)=Q(X,Y) (X-Y) (X-\zeta Y) \cdots \bigl(X-\zeta^{i-1}Y\bigr)=Q(X,Y) P_{i-1}(X,Y),$$
 since in particular $P\left(X,Y\right)$ vanishes on the closed subsets $\mathrm{Gr}(e), \cdots, \mathrm{Gr}\bigl(s^{i-1}\bigr)$, hence lies in $I_{i-1}$. Since $i<d$ we have that $P_{i-1}\bigl(\zeta^i, 1\bigr)\neq 0$, which implies that $Q\bigl(\zeta^i, 1\bigr)=0$ since $P$ has to vanish on $\mathrm{Gr}\bigl(s^i\bigr)$. Since both $P$ and $P_{i-1}$ are homogeneous, we have that $Q$ is homogeneous as well. Hence, writing $Q(X,Y)= \sum_{j} \alpha_j X^j Y^{k-j}$, we get that $Q(X, 1)=\sum_j \alpha_j X^j=\bigl(X-\zeta^i\bigr) Q_1(X)$ for some $Q_1\in \mathbb{C}[X]$ as $Q\bigl(\zeta^i, 1\bigr)=0$. But $Q$ is homogeneous, implying that $Q(X,Y)= \bigl(X-\zeta^i Y\bigr) Q_2(X,Y)$ for some $Q_2\in\mathbb{C}[X,Y]$, which concludes.    
\end{proof}

In particular, the kernel of the map 
\begin{equation*}
R\otimes_{\mathbb{C}} R\cong \mathcal{O}(V \times V)\twoheadrightarrow \mathcal{O}\bigl(s^{\leq i}\bigr)
\end{equation*}
 is generated by $P_i$ as a graded $R$--bimodule.

More generally the following holds:

\begin{coro}\label{identB}
Let $0\leq i \leq d-1$ and $0\leq j \leq d-1$. The $R$--bimodule $\mathcal{O}\bigl(s^{[i,i+j]}\bigr)$ can be identified with the quotient of the ring $\C[X,Y]$ by the homogeneous ideal $I_{i,j}$ generated by $P_j\bigl(\zeta^{-i}X,Y\bigr)$, where the left (resp. right) action of $R$ on $\mathcal{O}\bigl(s^{[i,i+j]}\bigr)$ corresponds to the multiplication by the image of $\mathbb{C}[X]$ (resp. the image of $\mathbb{C}[Y]$) in $\C[X,Y]/I_{i,j}$.
\end{coro}

\begin{proof}
This follows from the fact that $\mathcal{O}\bigl(s^{[i,i+j]}\bigr)\cong\mathcal{O}\bigl(s^i\bigr)\otimes_R\mathcal{O}\bigl(s^{\leq j}\bigr)$ (see Lemma~\ref{lem:isos}). The latter can be identified with the tensor product over $\C[Z]$ of the quotient of $\C[X,Z]$ by the ideal generated by $P_0\bigl(X,\zeta^{i}Z\bigr)=X-\zeta^{i}Z$ with the quotient of $\C[Z,Y]$ by the homogeneous ideal generated by $P_j(Z,Y)$ (see Lemma~\ref{generateur}). This tensor product is isomorphic as $\C[X,Y]$--module, hence as $R$--bimodule, to the quotient of $\C[X,Y]$ by the homogeneous ideal generated by $P_j\bigl(\zeta^{-i}X,Y\bigr)$.
\end{proof}

In particular, if $A\subseteq W$ is cyclically connected, then $A$ is of the form $s^{[i,i+j]}$ for some $0\leq i,j \leq d-1$ and we will simply denote by $I_A$ the ideal $I_{i,j}$ and by $P_A(X,Y)$ the generator $P_j\bigl(\zeta^{-i}X,Y\bigr)$ of this ideal.

Note that throughout the paper, we will often switch from one description of these bimodules to the other one without mentioning it. 

\begin{coro}\label{free_left}
The bimodule $\mathcal{O}\bigl(s^{\leq i}\bigr)$ is free of rank $i+1$ as a left $R$--module, with basis given by $\bigl\{1, Y, \cdots, Y^i\bigr\}$. 
\end{coro}

\begin{proof}
It follows immediately from the fact that $I_i$ is generated by $P_i$ that the above set generates $\mathcal{O}\bigl(s^{\leq i}\bigr)$ as a left $R$--module. Assume that $\sum_{j=0}^i  Q_j Y^j = 0$ in $\mathcal{O}\bigl(s^{\leq i}\bigr)$ for some $Q_j\in \mathbb{C}[X]$ and let us show that $Q_j=0$ for all $j$. Note that we can assume that the element is homogeneous, that is, that there is $k\in\mathbb{Z}_{\geq 0}$ such that $Q_j=\alpha_j X^{k-j}$ for some $\alpha_j\in\mathbb{C}$, for all $j$. It follows that the polynomial $P\in\mathbb{C}[X,Y]$ defined by 
$$\sum_{j=0}^i  \alpha_j X^{k-j} Y^j$$
 is in $I_i$, hence that $\alpha_j=0$ for all $j$, since $I_i$ is generated by $P_i$ which has a non-zero term of degree $i+1$ in $Y$ while $P$ has no term with power of $Y$ greater than $i$.    
\end{proof}

\section{Classification of indecomposable objects}

In this section, we classify the indecomposable bimodules in $\mathcal{B}_W$. 

\begin{prop}
Let $i\geq 1$. There is a short exact sequence of graded $R$--bimodules 
$$0\longrightarrow \mathcal{O}\bigl(s^{i}\bigr)[-2i]\longrightarrow \mathcal{O}\bigl(s^{\leq i}\bigr) \longrightarrow \mathcal{O}\bigl(s^{\leq i-1}\bigr)\longrightarrow 0,$$
where the surjective map is the restriction map $\pi$ and the injective map $\iota$, under the identification $R=\mathbb{C}[X]\cong \mathcal{O}\bigl(s^i\bigr)$ as left $R$--modules, is given by $r \mapsto r P_{i-1}$. 
\end{prop}

\begin{proof}
Since $\mathcal{O}\bigl(s^{ i}\bigr)$ is free as a left $R$--module, the map $\iota$ is a morphism of left $R$--modules. To show that it is a morphism of (graded) bimodules, it suffices to show, using the identification $\mathcal{O}\bigl(s^i\bigr)\cong R_{s^i}$, that $XP_{i-1}(X,Y) = P_{i-1}(X,Y) \zeta^{i} Y$ in $\mathcal{O}\bigl(s^{\leq i}\bigr)$. This holds since $P_i(X,Y)=P_{i-1}(X,Y) \bigl(X- \zeta^i Y\bigr) = 0$ in $\mathcal{O}\bigl(s^{\leq i}\bigr)$. Moreover, it is injective since, if $r\in \mathbb{C}[X]$ is such that $r P_{i-1}=0$ in $\mathcal{O}\bigl(s^{\leq i}\bigr)$, then $r P_{i-1}$ is a multiple of $P_i$ (viewed as homogeneous polynomials in $\mathbb{C}[X,Y]$), hence $r$ has to be a multiple of $\bigl(X-\zeta^i Y\bigr)$, but since $r\in\mathbb{C}[X]$ it implies that $r=0$. Therefore the map $\iota$ is injective and we have $\mathrm{im}(\iota)\subseteq \ker(\pi)$ since $P_{i-1}$ is zero in $\mathcal{O}\bigl(s^{\leq i-1}\bigr)$. It remains to show that $\ker(\pi)\subseteq\mathrm{im}(\iota)$. This follows again from the fact that $P_{i-1}(X,Y) \bigl(X- \zeta^i Y\bigr) = 0$ in $\mathcal{O}\bigl(s^{\leq i}\bigr)$.     
\end{proof}

\begin{lemma}\label{recPkgen}
Let $k, i$ and $j$ be integers in $\{0, \cdots, d-1\}$ with $k=i+j$. We have 
$$ P_k(X,Y) = c_k^i P_i(X,Z) + d_k^j P_j(Z,Y),$$
where $c_k^i, d_k^j \in\C[X,Y,Z]$ are defined as $c_k^k=d_k^k=1$ for all $0 \leq k \leq d-1$, as
$$ \begin{cases}
c_k^0 = P_{k-1}(X, \zeta Y) + \sum_{r=0}^{k-2}P_r(Z,Y) P_{k-2-r}\bigl(X,\zeta^{r+2} Y\bigr) + P_{k-1}(Z,Y), \\
d_k^0  =  P_{k-1}( X, \zeta Y) + \sum_{r=0}^{k-2}\zeta^{r+1} P_r(X,Z) P_{k-2-r}\bigl( X, \zeta^{r+2} Y \bigr) + \zeta^{k} P_{k-1}(X,Z), 
\end{cases}$$
for all $1 \leq k \leq d-1$, (with the convention that the sum is $0$ when $k=1$), and, for $0 <i,j < k$, by induction on $k$ by
$$ \begin{cases}
c_k^i  =  c_{k-1}^{i-1}  + c_{k-1}^{i}\zeta^{i}\bigl( Z -\zeta^{j} Y\bigr), \\
d_k^j  =  d_{k-1}^{j}\bigl(X- \zeta^{i} Z \bigr) + \zeta^{i}d_{k-1}^{j-1}.
\end{cases}$$
\end{lemma}

\begin{proof}
We argue by induction on $k$. It is clear that 
$$P_0(X,Y) = P_0(X,Z) + P_0(Z,Y)=c_0^0 P_0(X,Z) + d_0^0 P_0(Z,Y).$$
Let us now detail the two cases $(i,j) =( 1,0)$ and $(i,j) =( 0,1)$. Since 
\begin{equation}\label{expc10}
c_1^0 = P_0 ( X, \zeta Y) + P_0( Z, Y) = X +Z- (1+\zeta) Y
\end{equation}
and
\begin{equation}\label{expd10}
d_1^0 = P_0( X, \zeta Y) + \zeta P_0( X,Z) = (1+ \zeta)X- \zeta Z - \zeta Y,
\end{equation}
we get that
\begin{eqnarray*}
P_1(X,Y ) & = & (X-Y)(X- \zeta Y) \\
& = & (X -Z +Z-Y)(X- \zeta Y) \\
& = & (X- \zeta Y)P_0(X,Z) +  (Z-Y)(X- \zeta Y) \\
& = & (X- \zeta Y)P_0(X,Z) +  (Z-Y)(X-Z + Z -\zeta Y) \\
& = & (X- \zeta Y)P_0(X,Z) +  ( Z-Y)P_0(X,Z) +P_1(Z,Y) \\
& = & (X +Z- (1+\zeta) Y)P_0(X,Z) +P_1(Z,Y)\\
& = & c_1^0 P_0(X,Z) + d_1^1 P_1(Z,Y)
\end{eqnarray*}
and, similarly,
\begin{eqnarray*}
P_1(X,Y ) & = & (X-Y)(X- \zeta Y) \\
& = &P_1(X,Z) + ((1+ \zeta)X- \zeta Z - \zeta Y)P_0(Z,Y)\\
& = &c_1^1 P_1(X,Z)+ d_1^0 P_0(Z,Y).
\end{eqnarray*}

Now assume that $k>1$. Firstly, let us assume that $i,j\neq 0$. In that case, we argue by induction on $k$. We use twice the induction hypothesis on $P_{k-1}\left( X,Y \right)$ (for $(i-1,j)$ and $(i,j-1)$) to conclude that
\begin{eqnarray*}
P_k(X,Y) & = & P_{k-1}( X, Y)\bigl(X- \zeta^{i+j} Y\bigr) \\
& = &P_{k-1}( X, Y )\bigl(X- \zeta^{i} Z + \zeta^{i} Z -\zeta^{i+j} Y\bigr)  \\
& = &\bigl( c_{k-1}^{i-1} P_{i-1}( X, Z) + d_{k-1}^{j} P_{j}( Z, Y ) \bigr)\bigl(X- \zeta^{i} Z \bigr) \\
& & +\bigl( c_{k-1}^{i} P_{i}( X, Z) + d_{k-1}^{j-1} P_{j-1}( Z, Y) \bigr) \zeta^{i}\bigl( Z -\zeta^{j} Y\bigr)  \\
& = &\bigl(c_{k-1}^{i-1}  + c_{k-1}^{i}\zeta^{i}\bigl( Z -\zeta^{j} Y\bigr)\bigr)P_{i}( X, Z ) \\
& &+\bigl( d_{k-1}^{j}\bigl(X- \zeta^{i} Z \bigr) + \zeta^{i}d_{k-1}^{j-1}  \bigr) P_{j}( Z, Y)\\
& = & c_k^i P_{i}( X, Z) + d_k^j P_{j}( Z, Y),
\end{eqnarray*}
which concludes for $i,j \neq 0$. 

It remains to deal with the case where $i=0$ or $j=0$. To this end we first express $P_k( X ,Y)$, using the induction hypothesis, in terms of $P_0(X ,Z)$ and $P_k(Z ,Y)$:
\begin{eqnarray*}
P_k(X,Y) & = & P_{k-1}(X, Y)\bigl(X- \zeta^{k} Y\bigr) \\
& = &\bigl( c_{k-1}^{0} P_{0}( X, Z) + d_{k-1}^{k-1} P_{k-1}(Z, Y) \bigr)\bigl(X- \zeta^{k} Y \bigr)  \\
& = &c_{k-1}^{0} \bigl(X- \zeta^{k} Y \bigr) P_{0}( X, Z) + d_{k-1}^{k-1} P_{k-1}( Z, Y) \bigl(X- Z + Z -\zeta^{k} Y \bigr)  \\
& = &\bigl( c_{k-1}^{0} \bigl(X- \zeta^{k} Y \bigr) + d_{k-1}^{k-1} P_{k-1}( Z, Y) \bigr) P_{0}( X, Z) + d_{k-1}^{k-1} P_{k}( Z, Y).
\end{eqnarray*}
Since $d_{k-1}^{k-1} = 1 = d_{k}^{k}$, it suffices to show that 
$$c_k^0 = c_{k-1}^{0} \bigl(X- \zeta^{k} Y \bigr) + P_{k-1}( Z, Y)$$
in order to conclude. This holds since
\begin{eqnarray*}
& ~ & c_{k-1}^{0} \bigl(X- \zeta^{k} Y \bigr) + P_{k-1}( Z, Y) \\
& = & \left[ P_{k-2}( X, \zeta Y) + \sum_{r=0}^{k-3}P_r( Z, Y) P_{k-3-r} \bigl( X, \zeta^{r+2} Y \bigr) + P_{k-2} (Z, Y)\right] \bigl(X- \zeta^{k} Y \bigr) \\
& & + P_{k-1}( Z, Y) \\
& = & P_{k-1}\bigl( X, \zeta Y \bigr) + \sum_{r=0}^{k-2}P_r( Z, Y) P_{k-2-r} \bigl( X, \zeta^{r+2} Y \bigr) + P_{k-1}(Z,Y) .
\end{eqnarray*}

Finally, we express $P_k( X , Y)$, using the induction hypothesis, in terms of $P_k( X ,Z)$ and $P_0( Z , Y)$:
\begin{eqnarray*}
P_k(X, Y) & = & P_{k-1}( X, Y)\bigl(X- \zeta^{k} Y\bigr) \\
& = &\bigl( c_{k-1}^{k-1} P_{k-1}( X, Z) + d_{k-1}^{0} P_{0}( Z, Y) \bigr)\bigl(X- \zeta^{k} Y \bigr)  \\
& = & c_{k-1}^{k-1} P_{k-1}( X, Z) \bigl(X- \zeta^{k}Z + \zeta^{k}Z -\zeta^{k} Y \bigr) + d_{k-1}^{0} \bigl(X- \zeta^{k} Y \bigr) P_{0}( Z, Y) \\
& = & c_{k-1}^{k-1} P_{k}( X, Z) + \bigl(c_{k-1}^{k-1} \zeta^{k} P_{k-1}( X, Z) + d_{k-1}^{0} \bigl(X- \zeta^{k} Y \bigr) \bigr) P_{0}( Z, Y).
\end{eqnarray*}
Since $c_{k-1}^{k-1} = 1 = c_{k}^{k}$, it suffices to show that 
$$d_k^0 = \zeta^{k} P_{k-1}( X, Z) + d_{k-1}^{0} \bigl(X- \zeta^{k} Y \bigr)$$
in order to conclude. This holds since
\begin{eqnarray*}
& ~ &\zeta^{k} P_{k-1}( X, Z) + d_{k-1}^{0} \bigl(X- \zeta^{k} Y \bigr) \\
& = & \zeta^{k} P_{k-1}( X, Z) + \left[ P_{k-2}\bigl( X, \zeta Y \bigr) + \sum_{r=0}^{k-3}\zeta^{r+1} P_r( X,Z) P_{k-3-r} \bigl( X, \zeta^{r+2} Y \bigr) \right. \\
& &  + \zeta^{k-1} P_{k-2}(X,Z) \Bigg] \bigl(X- \zeta^{k} Y \bigr)  \\
& = & P_{k-1}\bigl( X, \zeta Y \bigr) + \sum_{r=0}^{k-2}\zeta^{r+1} P_r( X,Z) P_{k-2-r} \bigl( X, \zeta^{r+2} Y \bigr) + \zeta^{k} P_{k-1} (X,Z),
\end{eqnarray*}
which completes the proof.
\end{proof}

In the sequel, it will be useful to have a closed formula for these coefficients in the case where $k=i+1$.

\begin{lemma}\label{recPkexpl}
Let $0 \leq i \leq d-2$. A closed formula for the polynomials $c_{i+1}^i$ and $d_{i+1}^i$ in $\C[X,Y,Z]$ defined in Lemma~\ref{recPkgen} is given by
$$ \begin{cases}
c_{i+1}^i (X,Y,Z) = X + \bigl( \sum_{r=0}^{i} \zeta^{r}\bigr) Z - \bigl(\sum_{r=0}^{i+1} \zeta^{r}\bigr) Y,\\
d_{i+1}^i (X,Y,Z) = \bigl( \sum_{r=0}^{i+1} \zeta^{r}\bigr) X - \bigl(\sum_{r=1}^{i+1} \zeta^{r}\bigr) Z -  \zeta^{i+1} Y.
\end{cases}$$
\end{lemma}

\begin{proof}
In the proof of Lemma \ref{recPkgen}, we have already seen (\eqref{expc10} and \eqref{expd10}) that the formulas hold for $i=0$. We proceed by induction on $i$ using both the recursive relation on these coefficients and the fact that $c_i^i=d_i^i=1$. We get
\begin{eqnarray*}
c_{i+1}^i & = & c_i^{i-1} + c_i^i \zeta^i \bigl(Z-\zeta Y \bigr) \\
& = & X + \left( \sum_{r=0}^{i-1} \zeta^{r} \right) Z - \left( \sum_{r=0}^{i} \zeta^{r} \right) Y +  \zeta^i \bigl(Z-\zeta Y \bigr) \\
& = & X + \left( \sum_{r=0}^{i} \zeta^{r} \right) Z - \left( \sum_{r=0}^{i+1} \zeta^{r} \right) Y
\end{eqnarray*}
and
\begin{eqnarray*}
d_{i+1}^i & = & \zeta d_i^{i-1} + d_i^i \bigl(X-\zeta Z \bigr) \\
& = & \zeta \left( \left( \sum_{r=0}^{i} \zeta^{r} \right) X - \left( \sum_{r=1}^{i} \zeta^{r} \right) Z -  \zeta^{i} Y \right) +  \bigl(X-\zeta Z \bigr) \\
& = & \left( \sum_{r=0}^{i+1} \zeta^{r} \right) X - \left( \sum_{r=1}^{i+1} \zeta^{r} \right) Z -  \zeta^{i+1} Y.
\end{eqnarray*}
\end{proof}

In the next lemmas, we show some technical results on tensors products of some specific $R$--bimodules, which will enable us to deduce that tensor products of the form $\mathcal{O}(A)\otimes_R \mathcal{O}(B)$, where $A,B\subseteq W$ are cyclically connected, always decompose (as graded $R$--bimodules) in a direct sum of (possibly shifted) rings of regular functions on graphs of cyclically connected subsets. This will allow us to classify the indecomposable objects in $\mathcal{B}_W$.
																																								
\begin{prop}\label{inj1}
Let $1\leq i< d-1$. There is an injective homomorphism of graded $R$--bimodules 
$$\varphi: \mathcal{O}\bigl(s^{\leq i+1}\bigr)\longrightarrow \mathcal{O}\bigl(s^{\leq 1}\bigr)\otimes_R \mathcal{O}\bigl(s^{\leq i}\bigr)$$
induced by $1\mapsto 1\otimes 1$.      
\end{prop}

\begin{proof}
Let $f: R\otimes_{\mathbb{C}} R\cong\mathbb{C}[X,Y] \longrightarrow \mathcal{O}\bigl(s^{\leq 1}\bigr)\otimes_R \mathcal{O}\bigl(s^{\leq i}\bigr)$ be the map defined by $a\otimes b \mapsto \mathrm{pr}_1^\star(a)\otimes \mathrm{pr}_2^\star(b)$, where $\mathrm{pr}_1:\mathrm{Gr}(e,s)\rightarrow V$, $\mathrm{pr}_2: \mathrm{Gr}\bigl(e,s,\dots,s^i\bigr)\rightarrow V$ are the projections on the first factor (resp. on the second factor). It is clearly a homomorphism of graded $R$--bimodules. Using Corollary~\ref{identB}, we identify $\mathcal{O}\bigl(s^{\leq i+1}\bigr)$ with $\mathbb{C}[X,Y]$ modulo $P_{i+1}(X,Y)$ and $\mathcal{O}\bigl(s^{\leq 1}\bigr)\otimes_R\mathcal{O}\bigl(s^{\leq i}\bigr)$ with $\mathbb{C}[X,Y,Z]$ modulo $P_1(X,Z)$ and $P_i(Z,Y)$.

We show that the map $f$ factors through $\mathcal{O}\bigl(s^{\leq i+1}\bigr)$. This holds if and only if $f$~maps $P_{i+1}(X,Y)$ to zero. This is ensured by Lemma~\ref{recPkgen}, as 
$$P_{i+1}(X,Y)= c_{i+1}^1 (X,Y,Z) P_1(X,Z) + c_{i+1}^i (X,Y,Z) P_i(Z,Y).$$ 
Hence the map $f$ factors through $\mathcal{O}\bigl(s^{\leq i+1}\bigr)$, inducing a homomorphism of graded $R$--bimodules $\varphi: \mathcal{O}\bigl(s^{\leq i+1}\bigr)\longrightarrow \mathcal{O}\bigl(s^{\leq 1}\bigr)\otimes_R \mathcal{O}\bigl(s^{\leq i}\bigr)$. It remains to show that $\varphi$ is injective. To this end, consider the basis $\bigl\{1, Y, \dots, Y^{i+1}\bigr\}$ of $\mathcal{O}\bigl(s^{\leq i+1}\bigr)$ as left $R$--module (see Corollary~\ref{free_left}). A basis of $\mathcal{O}\bigl(s^{\leq 1}\bigr)\otimes_R \mathcal{O}\bigl(s^{\leq i}\bigr)$ as a left $R$--module is given by $\left\{1, Y, \dots, Y^i,Z, ZY, \dots, ZY^i\right\}$. Hence to show that the map is injective, it suffices to see that the elements $1, Y, \dots, Y^{i+1}$ of $\mathcal{O}\bigl(s^{\leq 1}\bigr)\otimes_R \mathcal{O}\bigl(s^{\leq i}\bigr)$ are $\mathbb{C}[X]$-linearly independent, that is, that $Y^{i+1}\in \mathcal{O}\bigl(s^{\leq i}\bigr)\otimes_R \mathcal{O}\bigl(s^{\leq i}\bigr)$ is not a $\mathbb{C}[X]$-linear combination of $1, Y, \dots, Y^i$. Assume that 
$$Y^{i+1}= \sum_{k=1}^{i+1} \alpha_k X^k Y^{i+1-k}$$
for some $\alpha_k\in \mathbb{C}$. This implies that $Y^{i+1}- \sum_{k=1}^{i+1} \alpha_k X^k Y^{i+1-k}=0$ in $\mathcal{O}\bigl(s^{\leq 1}\bigr)\otimes_R \mathcal{O}\bigl(s^{\leq i}\bigr)$, hence, that the polynomial $Y^{i+1}- \sum_{k=1}^{i+1} \alpha_k X^k Y^{i+1-k}\in\C[X,Y,Z]$ is equal to zero modulo $P_1(X,Z)$ and $P_i(Z,Y)$. It must therefore be of the form $Q_1 P_1(X,Z)+Q_2 P_i(Z,Y)$ for homogeneous polynomials $Q_1, Q_2\in\C[X,Y,Z]$. Since $P_i(Z,Y)$ has degree $i+1$, we must have $Q_2\in\C^\times$, and this implies that the monomial $ZY^i$ is contributed from $Q_2 P_i(Z,Y)$ with a non-zero coefficient (since it has a non-zero coefficient in $P_i(Z,Y)$ because $i<d-1$); but it cannot be contributed from $Q_1 P_1(X,Z)$ as $Q_1$ has degree $i-1$, a contradiction. Hence $Y^{i+1}$ is not a $\C[X]$--linear combination of $1, Y, \dots, Y^i$ in $\mathcal{O}\bigl(s^{\leq 1}\bigr)\otimes_R \mathcal{O}\bigl(s^{\leq i}\bigr)$, which completes the proof.     
\end{proof}

\begin{prop}\label{inj2}
Let $1\leq i< d-1$. There is an injective homomorphism of graded $R$--bimodules 
$$\psi: \mathcal{O}\bigl(s^{[1,i]}\bigr)[-2]\longrightarrow \mathcal{O}\bigl(s^{\leq 1}\bigr)\otimes_R \mathcal{O}\bigl(s^{\leq i}\bigr)$$
induced by $1 \mapsto m$ where 
$m  = \bigl( \sum_{r=0}^{i} \zeta^{r} \bigr) Z  -  \bigl( \sum_{r=0}^{i-1} \zeta^{r} \bigr)X  -  \zeta^{i} Y$,
under the identification of $\mathcal{O}\bigl(s^{\leq 1}\bigr)\ot_R \mathcal{O}\bigl(s^{\leq i}\bigr)$ with $\mathbb{C}[X,Y,Z]$ modulo $P_1(X,Z)$ and $P_i(Z,Y)$.      
\end{prop}

\begin{proof}
Using Corollary~\ref{identB}, we start by identifying $\mathcal{O}\bigl(s^{[1,i]}\bigr)$ with $\C[X,Y]$ modulo $P_{i-1}\bigl(\zeta^{-1}X,Y\bigr)$ and $\mathcal{O}\bigl(s^{\leq 1}\bigr)\otimes_R \mathcal{O}\bigl(s^{\leq i}\bigr)$ with $\mathbb{C}[X,Y,Z]$ modulo $P_1(X,Z)$ and $P_i(Z,Y)$. The homomorphism $\psi$ of graded $R$--bimodules sending $1$ to $m$ is well-defined if and only if $\psi$ maps $P_{i-1}\bigl(\zeta^{-1}X,Y\bigr)$ to zero. This holds as long as $m P_{i-1}\bigl(\zeta^{-1}X,Y\bigr)$ can be written as a $\mathbb{C}[X,Y,Z]$--linear combination of $P_1(X,Z)$ and $P_i(Z,Y)$. But by Lemma \ref{recPkgen}, we have
\begin{eqnarray*}
P_i (Z,Y) & = & c_i^1\bigl(Z, Y, \zeta^{-1}X\bigr) P_1\bigl(Z, \zeta^{-1}X\bigr) + d_i^{i-1} \bigl(Z, Y, \zeta^{-1}X\bigr)P_{i-1}\bigl( \zeta^{-1}X, Y\bigr) \\
& = &  \zeta^{-1} c_i^1\bigl(Z, Y, \zeta^{-1}X\bigr) P_1\bigl(X, Z\bigr) + d_i^{i-1} \bigl(Z, Y, \zeta^{-1}X \bigr)P_{i-1}\bigl( \zeta^{-1}X, Y\bigr), 
\end{eqnarray*} 
which concludes since
$$d_i^{i-1} \bigl(Z, Y, \zeta^{-1}X \bigr) = \left( \sum_{r=0}^{i} \zeta^{r} \right) Z - \left( \sum_{r=0}^{i-1} \zeta^{r} \right)X  -  \zeta^{i} Y=m$$
by Lemma \ref{recPkexpl}. It remains to show that the map is injective. A basis of $\mathcal{O}\bigl(s^{[1,i]}\bigr)$ as a left $R$--module is given by $\bigl\{1, Y, \dots, Y^{i-1}\bigr\}$. Indeed, since $\mathcal{O}\bigl(s^{[1,i]}\bigr)\cong \mathcal{O}(s)\otimes_R \mathcal{O}\bigl(s^{\leq i-1}\bigr)$, a basis of $\mathcal{O}\bigl(s^{[1,i]}\bigr)$ as a left $R$--module can be obtained by just taking a basis of $\mathcal{O}\bigl(s^{\leq i-1}\bigr)$ as a left $R$--module since the effect of tensoring by $\mathcal{O}(s)$ on the left just twists the multiplication, letting $X$ act by multiplication by $s^{-1}(X)=\zeta X$. Taking the image of this basis by $\psi$, we obtain $R$--linearly independent elements, since $\left\{1, Y, \dots, Y^i, Z, ZY,\dots, ZY^i\right\}$ is a basis of $\mathcal{O}\bigl(s^{\leq 1}\bigr)\otimes_R \mathcal{O}\bigl(s^{\leq i}\bigr)$ as left $R$--module. 
\end{proof}

Propositions~\ref{inj1} and~\ref{inj2} allow us to prove the following result, which is an analogue of~\cite[Proposition 4.6]{S} (where it is proven for dihedral groups) in our setting:

\begin{prop}[Soergel's Lemma]\label{decomp}
Let $1\leq i< d-1$. There is an isomorphism of graded $R$--bimodules 
$$\mathcal{O}\bigl(s^{\leq 1}\bigr)\otimes_R \mathcal{O}\bigl(s^{\leq i}\bigr) \cong \mathcal{O}\bigl(s^{\leq i+1}\bigr) \oplus \mathcal{O}\bigl(s^{[1,i]}\bigr)[-2].$$
\end{prop}

\begin{proof}
In the former Propositions \ref{inj1} and \ref{inj2}, we have proven that the two graded $R$--bimodules $\mathcal{O}\bigl(s^{\leq i+1}\bigr) $ and $ \mathcal{O}\bigl(s^{[1,i]}\bigr)[-2]$ embed, through the maps $\varphi$ and $\psi$, into $\mathcal{O}\bigl(s^{\leq 1}\bigr)\otimes_R \mathcal{O}\bigl(s^{\leq i}\bigr)$. But the latter is a free left $R$--module of rank $2(i+1)$ with basis $\left\{1, Y, \dots, Y^i, Z, ZY,\dots, ZY^i\right\}$. It possesses two $R$--subbimodules: $\mbox{im}(\varphi)$, which is free with basis $\left\{1, Y, \dots, Y^i\right\}$, and $\mbox{im}(\psi)$, which is free with basis $\left\{m, mY,\dots, mY^i\right\}$ where $m=\bigl( \sum_{r=0}^{i} \zeta^{r} \bigr) Z  -  \bigl( \sum_{r=0}^{i-1} \zeta^{r} \bigr)X  -  \zeta^{i} Y$. These two $R$--subbimodules generate $\mathcal{O}\bigl(s^{\leq 1}\bigr)\otimes_R \mathcal{O}\bigl(s^{\leq i}\bigr)$ as a left $R$--module since $\bigl( \sum_{r=0}^{i} \zeta^{r} \bigr) \neq 0$ for $1\leq i< d-1$. Both are free of rank $i+1$, so, by rank considerations, they form a direct sum decomposition of $\mathcal{O}\bigl(s^{\leq 1}\bigr)\otimes_R \mathcal{O}\bigl(s^{\leq i}\bigr)$ as a left $R$--module. But, since they are in fact, by Propositions \ref{inj1} and \ref{inj2}, $R$--subbimodules of $\mathcal{O}\bigl(s^{\leq 1}\bigr)\otimes_R \mathcal{O}\bigl(s^{\leq i}\bigr)$ (i.e. they are both stable by multiplication by $Y$), we get that 
$$\mathcal{O}\bigl(s^{\leq 1}\bigr)\otimes_R \mathcal{O}\bigl(s^{\leq i}\bigr) \cong \mathcal{O}\bigl(s^{\leq i+1}\bigr) \oplus \mathcal{O}\bigl(s^{[1,i]}\bigr)[-2]$$
not only as a left $R$--module but as a left $R$--bimodule.
\end{proof}

\begin{remark}\label{rmq:FS}
The surjective map $R\otimes_{\mathbb{C}} R\twoheadrightarrow \mathcal{O}(W) = \mathcal{O}\bigl(s^{\leq d-1}\bigr)$ factors through $R\otimes_{R^s} R$ (see Remark~\ref{indec}). Note that $R^s=\C[X^d]$. Comparing the graded dimensions using on one hand Lemma~\ref{free_left} and on the other hand the fact that, as an $R^s$--module, we have 
$$R= R^s \oplus R^s X \oplus \cdots \oplus R^s X^{d-1},$$
we get an isomorphism of graded $R$--bimodules
\begin{equation}\label{isoFS}
\mathcal{O}(W)\cong R\otimes_{R^s} R.
\end{equation}
This implies in particular that
$$\mathcal{O}(W) \otimes_R \mathcal{O}(W) \cong \mathcal{O}(W) \oplus \mathcal{O}(W)[-2] \oplus \cdots \oplus \mathcal{O}(W)[-2(d-1)].$$ 
Hence the Karoubi envelope of the additive monoidal category generated by $ \mathcal{O}(W) $ is a full subcategory of $\mathcal{B}_W$ which possesses, up to isomorphism and grading shifts, only two indecomposable objects, $R$ and $\mathcal{O}(W)$, as already mentioned in the introduction. 
\end{remark}

The next lemma is the analogue of \cite[Lemma 4.5 (1)]{S}:

\begin{lemma}\label{lem:FS}
There is an isomorphism of graded $R$--bimodules 
$$\mathcal{O}\bigl(s^{\leq 1}\bigr)\otimes_R \mathcal{O}(W)\cong \mathcal{O}(W)\oplus \mathcal{O}(W)[-2].$$ 
\end{lemma}

\begin{proof}
As a left $R$--module, a basis of $\mathcal{O}\bigl(s^{\leq 1}\bigr)$, identified with $\mathbb{C}[X,Z]$ modulo $P_{1}(X,Z)$, is given by $\{1, Z\}$. But since for all $a\in\mathcal{O}\bigl(s^{\leq 1}\bigr)$ and all $b\in R^s$, we have $ab=ba$, the two left $R$--submodules generated respectively by each of these two elements lie in $R$--$\mathrm{mod}$--$R^s$. Together with Isomorphism \eqref{isoFS}, it follows that 
$$\mathcal{O}\bigl(s^{\leq 1}\bigr)\otimes_R \mathcal{O}(W)\cong \mathcal{O}\bigl(s^{\leq 1}\bigr)\otimes_{R^s} R\cong (R\otimes_{R^s} R) \oplus (R \otimes_{R^s} R)[-2],$$
which concludes. 
\end{proof}

\begin{remark}\label{decompOWA}
In the exact same way, one can prove that there are isomorphisms of graded $R$--bimodules 
$$\mathcal{O}\bigl(s^{\leq i}\bigr) \otimes_R \mathcal{O}\bigl(s^{\leq 1}\bigr) \cong \mathcal{O}\bigl(s^{\leq i+1}\bigr) \oplus \mathcal{O}\bigl(s^{[1,i]}\bigr)[-2]$$
for all $1\leq i< d-1$, and
$$ \mathcal{O}(W) \otimes_R \mathcal{O}\bigl(s^{\leq 1}\bigr) \cong \mathcal{O}(W)\oplus \mathcal{O}(W)[-2].$$

In particular, it follows that
$$\mathcal{O}\bigl(s^{\leq 1}\bigr) \otimes_R \mathcal{O}\bigl(s^{\leq i}\bigr) \cong \mathcal{O}\bigl(s^{\leq i}\bigr) \otimes_R \mathcal{O}\bigl(s^{\leq 1}\bigr)$$
for all $1\leq i \leq d-1$; which, together with Corollary \ref{cor:isos}, also implies that
$$\mathcal{O}\bigl(s^{\leq 1}\bigr)\otimes_R \mathcal{O}(A)\cong \mathcal{O}(A) \otimes_R \mathcal{O}\bigl(s^{\leq 1}\bigr)$$
for every cyclically connected subset $A\subseteq W$.
\end{remark}

Recall that $\mathcal{B}_W$ is the category which is generated (as additive, graded, monoidal, Karoubian and stable by grading shifts category) by $\mathcal{O}\bigl(s^{\leq 1})=\mathcal{O}(e,s)$. We now have all the required tools to classify the indecomposable objects in $\mathcal{B}_W$:

\begin{thm}[Classification of indecomposable objects in $\mathcal{B}_W$]\label{classif}
The indecomposable objects in $\mathcal{B}_W$ are, up to isomorphism and grading shift, given by $$\left\{ \mathcal{O}(A)~|~A\subseteq W \text{ is cyclically connected} \right\}.$$
\end{thm}

\begin{proof}
Firstly, recall from Remark~\ref{indec} that $\mathcal{O}(A)$ is indecomposable as graded $R$--bimodule whenever $A\subseteq W$. So we just need to check that
\begin{enumerate}
\item\label{item1classif} for any two cyclically connected subsets $A, B\subseteq W$, the graded $R$--bimodule $\mathcal{O}(A) \otimes_R \mathcal{O}(B)$ decomposes in a direct sum of various (possibly shifted) $\mathcal{O}(C)$, with all $C\subseteq W$ cyclically connected;
\item\label{item2classif} for any cyclically connected subset $A\subseteq W$, the graded $R$--bimodule $\mathcal{O} (A)$ belongs to $\mathcal{B}_W$, i.e., it occurs as an indecomposable summand of some tensor power of $\mathcal{O}\bigl(s^{\leq 1}\bigr)$. 
\end{enumerate}

Let $A, B\subseteq W$ be cyclically connected. By definition, there are $i,j,k,l \in \left\{0,\cdots,d-1\right\}$ such that $A = \{s^i, s^{i+1}, \dots, s^{i+j}\}$ and $B = \{s^k, s^{k+1}, \dots, s^{k+l}\}$. Lemma~\ref{lem:isos} implies that
\begin{equation}\label{isocycs}
\mathcal{O} (A) \otimes_R \mathcal{O}(B) \cong \mathcal{O} \bigl(s^{i}\bigr) \otimes_R \mathcal{O} \bigl(s^{\leq j}\bigr) \otimes_R \mathcal{O} \bigl(s^{\leq l}\bigr) \otimes_R \mathcal{O} \bigl(s^{k}\bigr).
\end{equation}
We prove \eqref{item1classif} by induction on $|B|=l+1$. If $l=0$, the isomorphism \eqref{isocycs} reduces to
\begin{equation*}
\mathcal{O} (A) \otimes_R \mathcal{O}(B ) \cong \mathcal{O} \bigl(s^{i}\bigr) \otimes_R  \mathcal{O} \bigl(s^{\leq j}\bigr) \otimes_R \mathcal{O} \bigl(s^{k}\bigr) \cong \mathcal{O} \bigl(s^{[i+k,i+k+j]}\bigr),
\end{equation*}
where the last isomorphism follows from Lemma \ref{lem:isos} again. Hence \eqref{item1classif} holds in that case. Now let $l\geq 1$. If $j=0$, then \eqref{item1classif} can be proven exactly as in the base case of the induction. So suppose that $1 \leq j \leq d-1$. By Proposition \ref{decomp}, we know that $\mathcal{O} (A) \otimes_R \mathcal{O}(B )$ is a direct summand of 
$$\mathcal{O} \bigl(s^{i}\bigr) \otimes_R \mathcal{O} \bigl(s^{\leq j}\bigr) \otimes_R \mathcal{O} \bigl(s^{\leq 1}\bigr) \otimes_R \mathcal{O} \bigl(s^{\leq l-1}\bigr) \otimes_R \mathcal{O} \bigl(s^{k}\bigr),$$
therefore, by the Krull-Schmidt property, it suffices to show that this graded $R$--bimodule is a direct sum of various $\mathcal{O}(C)$ with $C\subseteq W$ cyclically connected. But it decomposes, using Remark \ref{decompOWA}, as
\begin{multline*}
\mathcal{O} \bigl(s^{i}\bigr) \otimes_R \mathcal{O} \bigl(s^{\leq j+1}\bigr)  \otimes_R \mathcal{O} \bigl(s^{\leq l-1}\bigr) \otimes_R \mathcal{O} \bigl(s^{k}\bigr)\\
\oplus \mathcal{O} \bigl(s^{i+1}\bigr) \otimes_R \mathcal{O} \bigl(s^{\leq j-1}\bigr) \otimes_R \mathcal{O} \bigl(s^{\leq l-1}\bigr) \otimes_R \mathcal{O} \bigl(s^{k}\bigr)[-2]
\end{multline*}
if $j<d-1$, or as
\begin{multline*}
\mathcal{O} \bigl(s^{i}\bigr) \otimes_R \mathcal{O} \bigl(s^{\leq d-1}\bigr)  \otimes_R \mathcal{O} \bigl(s^{\leq l-1}\bigr) \otimes_R \mathcal{O} \bigl(s^{k}\bigr)\\
\oplus \mathcal{O} \bigl(s^{i}\bigr) \otimes_R \mathcal{O} \bigl(s^{\leq d-1}\bigr) \otimes_R \mathcal{O} \bigl(s^{\leq l-1}\bigr) \otimes_R \mathcal{O} \bigl(s^{k}\bigr)[-2]
\end{multline*}
if $j=d-1$. In either case, we can apply the induction hypothesis to both of the terms appearing in the decomposition, and this concludes the proof of~\eqref{item1classif}. 

To prove \eqref{item2classif}, we first observe that $\mathcal{O} (s) \in \mathcal{B}_W$: this follows from Proposition \ref{decomp} with $i=1$. We then proceed by induction on~$|A|$. If $|A|=1$, then there is $i\in \{0, \cdots, d-1\}$ such that $A=\{s^i\}$, and $\mathcal{O}\bigl(s^i\bigr)=\mathcal{O}(s)^{\otimes i}$ by \eqref{isostwists}. As $\mathcal{O}(s)\in\mathcal{B}_W$, this implies that $\mathcal{O}(A)\in \mathcal{B}_W$. 
If $|A|>1$, there are $i\in \{0, \cdots, d-1\}$ and $j\in \{1, \cdots, d-1\}$ such that $A = \{s^i, s^{i+1}, \dots, s^{i+j}\}$; hence, by Lemma \ref{lem:isos}, we have 
$$ \mathcal{O} (A) \cong \mathcal{O} \bigl(s^i\bigr) \otimes_R  \mathcal{O} \bigl(s^{\leq j}\bigr)$$ 
which, by Proposition \ref{decomp}, is a direct summand of $\mathcal{O} \bigl(s^i\bigr) \otimes_R  \mathcal{O} \bigl(s^{\leq 1}\bigr) \otimes_R  \mathcal{O} \bigl(s^{\leq j-1}\bigr)$. But by induction, the three factors of this tensor product lie in $\mathcal{B}_W$. This completes the proof of \eqref{item2classif} as $\mathcal{B}_W$ is Karoubian.
\end{proof}

We do not know how to parametrize the indecomposable objects in the bigger category generated by the $\mathcal{O}\bigl(e, s^i\bigr)$, which are attached to subsets which are not cyclically connected except for $i=1$ and $i=-1$. Since $s^i$ is a reflection whenever $i$ is such that $s^i\neq e$, it could make sense as well to take this category as the analogue of the Soergel category. It seems however that the Grothendieck ring will be much bigger, and it is not even clear that it has finite rank as a $\mathbb{Z}[v, v^{-1}]$--module.   

\section{Homomorphisms between indecomposable objects}

In this section, we give a basis for the morphism spaces between indecomposable objets in $\mathcal{B}_W$. Given $B, B'\in\mathcal{B}_W$ we set 
$$\mathrm{Hom}(B, B')=\bigoplus_{i\in\mathbb{Z}} \mathrm{Hom}_{\mathcal{B}_W} (B, B'[i]).$$
Recall that morphisms in $\mathcal{B}_W$ are morphisms of graded $R$-bimodules (i.e., degree zero maps). 

\begin{prop}
Let $A, B$ be cyclically connected subsets of $W$. Let $g_{A,B}:= \prod_{\left\{i~|~s^i \in B\backslash A\right\}} (X-\zeta^i Y)$, which we view inside $\mathcal{O}(B)$. Then $\mathrm{Hom}(\mathcal{O}(A), \mathcal{O}(B))$ is a free right $R$--module of rank $|A\cap B|$ with basis given by the maps which send $1$ to $$g_{A,B}, X g_{A,B},\dots, X^{|A\cap B|-1}g_{A,B}.$$ 
\end{prop}

\begin{proof}
Note that since $\mathcal{O}(A)$ is cyclic, every homomorphism of bimodules $\mathcal{O}(A)\rightarrow \mathcal{O}(B)$ is fully determined by its value on $1$. Let $Q(X,Y)$ be any homogeneous polynomial in $\mathbb{C}[X,Y]$. The map $1\mapsto Q(X,Y)$ defines a homomorphism of graded $R$--bimodules if and only if $Q(X,Y)$ is killed by the action of $P_A(X,Y)$, that is, if and only if $Q(X,Y)P_A(X,Y)$ is a multiple of $P_B(X,Y)$. 

It follows that $Q(X,Y)$ has to be a multiple of $g_{A,B}$ (viewed as a polynomial in $\mathbb{C}[X,Y]$). Hence $\mathrm{Hom}(\mathcal{O}(A), \mathcal{O}(B))$ is generated (as bimodule) by $g_{A,B}$. Now we have 
$$P_B(X,Y)=g_{A,B} \left(\prod_{\left\{i~|~s^i \in A\cap B\right\}} (X-\zeta^i Y)\right),$$
and since $P_B(X,Y)$ generates the ideal $I_B$, we have that the elements $g_{A,B}, X g_{A,B},\dots, X^{|A\cap B|-1}g_{A,B}$ are $\mathbb{C}[Y]$-linearly independent in $\mathcal{O}(B)$ and that every polynomial $Q(X,Y)$ which is a multiple of $g_{A,B}$ has an image in $\mathcal{O}(B)$ which is a $\mathbb{C}[Y]$--linear combination of these elements. Indeed, since $P_B(X,Y)=0$ in $\mathcal{O}(B)$, the polynomial $X^{|A\cap B|} g_{A,B}$, and hence all $X^{i} g_{A,B}$ for any $i\geq |A\cap B|$, can be expressed as a $\mathbb{C}[Y]$--linear combination of the elements above. 
\end{proof}

\section{Presentations by generators and relations of the Grothendieck ring}

Recall that $A_W$ denotes the $\mathbb{Z}[v, v^{-1}]$--algebra defined as the split Grothendieck ring of $\mathcal{B}_W$. It follows immediately from Theorem \ref{classif} and the fact that there are $d (d-1)+1$ cyclically connected subsets of $W$ that $A_W$ is a free $\mathbb{Z}[v, v^{-1}]$--module of rank $d (d-1)+1$. Abusing notation and writing $s$ for $\langle \mathcal{O}(s) \rangle$ and $C_i$ for $\langle\mathcal{O}\bigl( s^{\leq i} \bigr)[i]\rangle$ ($1\leq i \leq d-1$), we get the following presentation of $A_W$ by generators and relations:

\begin{prop}[Presentation of the Grothendieck ring]\label{pres1}
The algebra $A_W$ is generated by $s$ and $C_1, \dots, C_{d-1}$ with relations
$$\begin{cases}
s^d=1,\\
C_i C_j=C_j C_i \ \forall i,j\text{ and } s C_i= C_i s \ \forall i,\\
C_1 C_i=C_{i+1} + s C_{i-1} \ \forall i=1,\dots, d-2,\\
C_1 C_{d-1}= \bigl(v+v^{-1}\bigr) C_{d-1},\\
s C_{d-1}= C_{d-1},
\end{cases}$$
with the convention that $C_0 := 1$. In particular, $A_W$ is commutative. 
\end{prop}

\begin{proof}
Every indecomposable object of $\mathcal{B}_W$ is, up to grading shift, of the form $\mathcal{O}(A)$ for some cyclically connected $A\subseteq W$. By definition, it means that $A = s^{[i,i+j]}$ for some $i,j\in \left\{0,\cdots,d-1\right\}$ and we have that $\langle \mathcal{O}(A)[j] \rangle = s^iC_j$. Hence the algebra $A_W$ is generated over $\mathbb{Z}[v, v^{-1}]$ by $s$ and $C_1, \dots C_{d-1}$.

We now show that the above relations hold in $A_W$. The first relation follows from~\eqref{isostwists} and $\mathcal{O}(e)\cong R$. The third and fourth relations are given respectively by Proposition~\ref{decomp} and Lemma~\ref{lem:FS}. The last relation follows from~\eqref{isosFS}. To prove the second set of relations, let us first observe that $s$ is central in $A_W$ (by Corollary~\ref{cor:isos}) and so is $C_1$ (by Remark~\ref{decompOWA}). Moreover, it follows from the third relation and the fact that $C_1$ and $s$ commute to each other that every $C_i$ can be expressed as a polynomial in $C_1$ and $s$. Every $C_i$ is then central and the algebra $A_W$ is commutative. Hence all the above relations hold in $A_W$, implying that we have a surjective map from the algebra defined by the above presentation to $A_W$. 

Note that $A_W$ has a $\mathbb{Z}\bigl[v, v^{-1}\bigr]$-basis consisting of the $\langle \mathcal{O}\bigl( s^{[i,i+j]} \bigr)[j]\rangle$, for $0\leq i \leq d-1$ and $0\leq j \leq d-2$, and $\langle \mathcal{O}(W)[d-1]\rangle$. Therefore, in order to complete the proof, it suffices to show that the above defined commutative algebra is $\mathbb{Z}\bigl[v, v^{-1}\bigr]$--linearly spanned by the $s^i C_j$ for $(i,j) \in \Sigma $ where 
$$\Sigma= \{0, \dots, d-1\} \times \{0, \dots, d-2\} \cup \{0\}\times\{d-1\}.$$
As noted above, every $C_i$ is a polynomial in $C_1$ and $s$. Consequently, we only need to show that the former statement holds for monomials of the form $s^kC_1^l$, which we do inductively. Since it is trivially satisfied for $s$ and $C_1$, we just need to check that the $\mathbb{Z}[v, v^{-1}]$--span $\langle s^i C_j~|~(i,j) \in \Sigma \rangle$ is stable under multiplication by $s$ and $C_1$. For $s$, this follows from the first and last relations, while for $C_1$, this follows from first, third and fourth relations.
\end{proof}

\begin{remark}\label{remA1}
Remember that we supposed from the beginning that $d>2$. Indeed, in the case where $d=2$, the category $\mathcal{B}_W$ as we defined it above would be Soergel's category of type $A_1$, whose Grothendieck ring has rank $2$. But in that case $\mathcal{O}(s)$ does not appear as a direct summand of a tensor power of $\mathcal{O}\bigl(s^{\leq 1}\bigr)$. Nonetheless, if we consider the monoidal category generated by both $\mathcal{O}\bigl(s^{\leq 1}\bigr)$ and $\mathcal{O}(s)$, we get a Grothendieck ring of rank $3$ with the same presentation as the one given in Proposition~\ref{pres1}. This could be an indication that, in the Coxeter case, the above category should be thought of as the category generated by Soergel bimodules and the bimodules $\mathcal{O}(x)$, $x\in W$. In type $A_2$, this category was investigated by the authors in~\cite{GT} and gives rise to a Grothendieck ring of rank $25$.  
\end{remark}

The presentation obtained in Proposition \ref{pres1} can be reduced to a presentation with only two generators:

\begin{prop}[Second presentation of the Grothendieck ring]\label{pres2}
The algebra $A_W$ is generated by $s$ and $C$ with relations
$$\begin{cases}
s^d=1,\\
sC = Cs,\\
sC_{d-1}=C_{d-1},\\
C C_{d-1}=\bigl(v+v^{-1}\bigr) C_{d-1},
\end{cases}$$
where $C_{d-1}:=\sum_{i=0}^{\lfloor\frac{d-1}{2}\rfloor} {{d-1-i}\choose{i}} \bigl(-s\bigr)^i C^{d-1-2i}$.
\end{prop}

\begin{proof}
For $1\leq k \leq d-1$, recall that we denote by $C_k$ the class of $ \mathcal{O}\bigl(s^{\leq k}\bigr)[k]$ and by $s$ the class of $\mathcal{O}(s)$ in the Grothendieck ring. Let $C:=C_1$. Note that the proof will in particular show that $C_{d-1}$, which was previously defined as the class of $\mathcal{O}(W)[d-1]$, can be expressed in terms of $s$ and $C$ as claimed above. More generally, we claim that
$$C_k=\sum_{i=0}^{\lfloor\frac{k}{2}\rfloor} {{k-i}\choose{i}} (-s)^i C^{k-2i},$$
for all $1\leq k \leq d-1$. For $k=1$, the formula trivially holds and for $k=2$, Proposition~\ref{decomp} implies that $C_2=C^2- s$. Assume that the claim holds for all $j\leq k-1$. Applying again Proposition~\ref{decomp}, we get by induction that

\begin{eqnarray*}
C_k & = & C C_{k-1} - s C_{k-2}\\
& = & \sum_{i=0}^{\lfloor\frac{k-1}{2}\rfloor} {{k-1-i}\choose{i}} (-s)^i C^{k-2i}+\sum_{i=0}^{\lfloor\frac{k-2}{2}\rfloor} {{k-2-i}\choose{i}} (-s)^{i+1} C^{k-2(i+1)}\\
& = & \sum_{i=0}^{\lfloor\frac{k-1}{2}\rfloor} {{k-1-i}\choose{i}} (-s)^i C^{k-2i}+\sum_{i=1}^{\lfloor\frac{k}{2}\rfloor} {{k-1-i}\choose{i-1}} (-s)^{i} C^{k-2i}.\\
\end{eqnarray*}
Using Pascal's rule ${{k-1-i}\choose{i}}+{{k-1-i}\choose{i-1}}={{k-i}\choose{i}}$, we can rewrite this equality as follows: if $k$ is odd then $\lfloor\frac{k-1}{2}\rfloor = \lfloor\frac{k}{2}\rfloor$ and we get
\begin{eqnarray*}
C_k = C^k+\Biggl(\sum_{i=1}^{\lfloor\frac{k}{2}\rfloor} {{k-i}\choose{i}} (-s)^i C^{k-2i}\Biggr) = \sum_{i=0}^{\lfloor\frac{k}{2}\rfloor} {{k-i}\choose{i}} (-s)^i C^{k-2i},
\end{eqnarray*} 
which concludes in that case; if $k$ is even then $\lfloor\frac{k-1}{2}\rfloor = \lfloor\frac{k}{2}\rfloor-1$ but since, for $i=\frac{k}{2}$, we have ${{k-1-i}\choose{i-1}}=1={{k-i}\choose{i}}$, we also get the claim. 

As a consequence, we can apply Tietze transformations to the presentation obtained in Proposition~\ref{pres1} to obtain an equivalent presentation of $A_W$. First we can remove from the one of Proposition~\ref{pres1} all generators except $s$ and $C_1$, but each removed generator has to be replaced in the relations with its equivalent word in $s$ and $C_1$. Then we want to remove some of the relations which can be derived from the others. Notice that, among the second set of relations, all are a consequence of the commutativity of $s$ and $C_1$, and so is the third relation. Hence they can be removed. The other ones cannot, and yield the last two relations of the presentation above.

\end{proof}

\begin{remark}\label{remBoRo} As an immediate corollary, we have that the Grothendieck ring $\text{Gr}^{\text{st}_{\text{B}}}(D(B))$ of the category $D(B)-\text{stab}_{\text{B}}$ studied in \cite[Section 5B]{BoRo} (which is a quotient of the stable category of the module category of the Drinfeld quantum double of the Taft algebra $B$) is isomorphic to a quotient of $A_W$. More precisely, we have the following isomorphism of $\mathbb{Z}[v, v^{-1}]$--algebras:
$$ \mathbb{Z}[v, v^{-1}] \otimes_{\Z} \text{Gr}^{\text{st}_{\text{B}}}(D(B)) \cong A_W / I $$
where $I$ is the ideal generated by $C_{d-1}$ and $C_{l-1} + s^{l} C_{d-l-1}$ for all $1 \leq l \leq d-1$. In the notation of~\cite{BoRo}, this isomorphism sends $\text{\bf v}_{\zeta}^{\text{st}_{\text{B}}}$ to $s$ and $\text{ \bf m}_{2}^{\text{st}_{\text{B}}}$ to $C_1$.
\end{remark}

\begin{coro}\label{cor:ext}
Let $T_s$ be defined as $C=T_s+v$. The quotient algebra $A_w/(s=1)$ is generated by a single element $T_s$ and a single relation of the form 
$$T_s^d + a_{d-1} T_s^{d-1} + \dots + a_1 T_s + a_0=0$$
for some $a_i\in\mathbb{Z}[v, v^{-1}]$. In particular, after localization by $a_0$ if $a_0\notin\mathbb{Z}[v, v^{-1}]^\times$, we get the (generic) Hecke algebra (with one parameter). 
\end{coro}

\begin{proof}
Setting $s=1$ in the presentation from Proposition~\ref{pres2} of $A_W$, all relations except the last one become trivial. Using the fact that, in the quotient algebra, $C_{d-1}=\sum_{i=0}^{\lfloor\frac{d-1}{2}\rfloor} {{d-1-i}\choose{i}}(-1)^i ( T_s + v)^{d-1-2i}$, this last relation $C C_{d-1}= (v+v^{-1}) C_{d-1}$ becomes
$$(T_s - v^{-1})\left(\sum_{i=0}^{\lfloor\frac{d-1}{2}\rfloor} {{d-1-i}\choose{i}} (-1)^i( T_s + v)^{d-1-2i}\right) = 0.$$
This is a polynomial relation of order $d$ in $T_s$, where the coefficient of $T_s^d$ is invertible. Hence, after localization by its constant coefficient, we recover the defining relation (with coefficient parameters) of the generic Hecke algebra (see \cite[Section 2.1]{Ma}) specialized to one parameter. In Section~\ref{ss} we will see how to factor this polynomial over $\C$ using Chebyshev polynomials.   
\end{proof}

\section{Semisimplicity}\label{ss}

In this section, we show that the algebra $A_W^{\mathbb{C}}$ defined by the same presentation as the one in Proposition~\ref{pres1} but over the complex numbers, with $v\in\mathbb{C}^\times$, is generically semisimple. Define polynomials $Q_i(X,Y)$ recursively by $Q_0=1$, $Q_1=X$ and
$$Q_{i+1}= X Q_i - Y Q_{i-1}.$$ 
It then follows from Proposition \ref{pres1} that $C_i= Q_i(C_1, s)$. Let $\eta\in\mathbb{C}^{\times}$ be such that $\eta^d=1$. The polynomials in $X$ given by $Q_i(X,\eta)$ will turn out to play an important role in this section. Set $U_i(X):=\sqrt{\eta}^{-i} Q_i(2\sqrt{\eta}X,\eta)$. We then get $U_0=1$, $U_1=2X$ and it follows from the recursive relation on the $Q_i$ that 
$$U_{i+1}= 2X U_i - U_{i-1}.$$
That is, the polynomials $U_i$ are Chebyshev polynomials of the second kind and, as such, $U_i$ has $i$ distinct roots, given by $\cos\bigl(\frac{k\pi}{i+1}\bigr)$ for $k=1, \dots, i$. As a consequence, the roots of $Q_i(X, \eta)$ are given by $2\sqrt{\eta} \cos\bigl(\frac{k\pi}{i+1}\bigr)$ for $k=1, \dots, i$. In particular, we have that

\begin{lemma}\label{lem:no_mult}
The polynomial $Q_i(X, \eta)\in\mathbb{C}[X]$ has no multiple root.
\end{lemma} 

We first need to show that the $\mathbb{C}$-algebra $A_W^\mathbb{C}$ keeps some of the properties of $A_W$ which we observed: 

\begin{prop}
We have $\dim(A_W^\mathbb{C})=d (d-1) +1$. A basis is given by the $s^i C_j$, with $(i,j)\in\Sigma$. 
\end{prop}

\begin{proof}
It is clear from the relations in Proposition~\ref{pres1} that the $s^i C_j$ with $(i,j)\in\Sigma$ linearly span $A_W^{\mathbb{C}}$. Arguing as in the proof of Proposition~\ref{pres2}, we see that the presentations from Propositions~\ref{pres1} and~\ref{pres2} are still equivalent. Keeping the same notation as before, we have $C_j=Q_j(C,s)$ for all $0\leq j\leq d-1$. 

To show that the above elements are linearly independent, we begin by constructing an action of $A_W^{\mathbb{C}}$ on the vector space $U:=\bigoplus_{(i,j)\in\Sigma} \mathbb{C} E_{i,j}$. For simplicity we write $E_{d-1}:=E_{0,d-1}$. We let $C$ act on the basis elements by 
\begin{multline*}
C E_{i,j}= E_{i,j+1} + E_{i+1,j-1}\text{  if  }1\leq j \leq d-2, \\
C E_{i,0}=E_{i,1} \text{  and  } C E_{d-1}=(v+v^{-1}) E_{d-1}
\end{multline*}
and we let $s$ act by $$s E_{i,j}=E_{i+1,j}\text{ if }0\leq j \leq d-2,~s E_{d-1}=E_{d-1},$$
where the indices $i$ in $E_{i,j}$ must be read modulo $d$.
We have to check that this defines an action of $A_W^\mathbb{C}$. To this end, we show that the linear transformations of $U$ defined by these actions satisfy the relations given in Proposition~\ref{pres2}. It is clear that the action of $s$ commutes with the action of $C=C_1$; in particular, the element $C_{d-1}$ as defined in Proposition~\ref{pres2} and $Q_{d-1}(C,s)$ (which was defined inductively) must act by the same linear transformation on $U$ (and more generally, the same holds for $C_j$ and $Q_j(C,s)$). It is also clear that the action of $s$ defines an automorphism of order $d$ of $U$. Hence it remains to show that the actions of $C_{d-1}$ and of $s C_{d-1}$ coincide, and that the actions of $C C_{d-1}$ and $(v+v^{-1}) C_{d-1}$ also coincide. We claim that for all $1\leq j\leq d-1$,  
\begin{equation}\label{eq:E00}
C_j E_{0,0}=E_{0,j}.
\end{equation}
If $j=1$ the claim holds by definition of the action of $C$ since, in that case, we have $Q_j(C,s)=C$. For $j\geq 1$, we get by induction that 
\begin{eqnarray*}
C_{j+1} E_{0,0} &=& C C_j E_{0,0} - s C_{j-1} E_{0,0}= C E_{0,j} - s E_{0, j-1}\\
&=& E_{0,j+1} + E_{1,j-1} - E_{1, j-1}=E_{0,j+1}.
\end{eqnarray*}
We now consider the action of the polynomial in $C$ and $s$ given by $C_{d-1}=Q_{d-1}(C,s)$. We claim that $C_{d-1} E_{i,j}= [v]_j E_{d-1},$ where $[v]_j=v^{-j} + v^{-j+2} +\dots + v^j$. We show it by induction on $j$. If $j=0$ then, using~\eqref{eq:E00}, we get 
$$C_{d-1} E_{i,0}=C_{d-1} s^i E_{0,0}=s^i C_{d-1} E_{0,0}= s^i E_{d-1}= E_{d-1}.$$
Similarly for $j=1$, we get that 
$$C_{d-1} E_{i,1}=C_{d-1} C E_{i,0}=C C_{d-1} E_{i,0}= C E_{d-1}=(v+v^{-1}) E_{d-1}=[v]_1 E_{d-1}.$$ 
Now let $j\geq 1$. We have
\begin{eqnarray*}
C_{d-1} E_{i,j+1} &=& C_{d-1} ( C E_{i,j}- E_{i+1,j-1})= C C_{d-1} E_{i,j}-[v]_{j-1} E_{d-1} \\
&=& ( (v+v^{-1}) [v]_j - [v]_{j-1}) E_{d-1}=[v]_{j+1} E_{d-1},
\end{eqnarray*}   
which shows the claim. Now, for all $(i,j)\in\Sigma$, we have that $C_{d-1} E_{i,j}=[v]_j E_{d-1}$, which implies that 
$$s C_{d-1} E_{i,j}= s [v]_j E_{d-1} = [v]_j E_{d-1} = C_{d-1} E_{i,j}$$
and 
$$C C_{d-1} E_{i,j}= C [v]_j E_{d-1} = [v]_j C E_{d-1} = (v+v^{-1}) [v]_j E_{d-1}=(v+v^{-1}) C_{d-1} E_{i,j}.$$
This shows that the last two relations of Proposition~\ref{pres2} are satisfied, hence $U$ is an $A_W^{\mathbb{C}}$-module. Now let $\sum_{(i,j)\in \Sigma} \alpha_{(i,j)} s^i C_j=0,$
where $\alpha_{(i,j)}\in\mathbb{C}$. We have by~\eqref{eq:E00} that $$0= \biggl(\sum_{(i,j)\in \Sigma} \alpha_{(i,j)} s^i C_j\biggr)\cdot E_{0,0} =\sum_{(i,j)\in\Sigma} \alpha_{(i,j)} E_{i,j},$$
which implies, since $\{E_{i,j}\}_{(i,j)\in\Sigma}$ forms a basis of $U$, that $\alpha_{(i,j)}=0$ for all $(i,j)\in\Sigma$ and therefore that $\{s^i C_j\}_{(i,j)\in\Sigma}$ forms a basis of $A_W^{\mathbb{C}}$. 
\end{proof}

To show that $A_W^{\mathbb{C}}$ is generically semisimple, we need to show that the regular module $_{A_W^{\mathbb{C}}}A_W^{\mathbb{C}}$ decomposes as a direct sum of simple $A_W^{\mathbb{C}}$-submodules. Since the algebra is commutative, every simple module is one-dimensional. The reflection $s$ has to act on any one-dimensional submodule by multiplication by a scalar $\eta\in \mathbb{C}$ such that $\eta^d=1$. Let $S(\eta):=1+\eta^{-1} s+\dots+\eta^{-d+1} s^{d-1}\in A_W^{\mathbb{C}}$ and consider, for $i=0, \dots, d-2$, the element $D_i^{\eta}:=S(\eta) C_{i}$ (recall that $C_0 = 1$). We set $D_{d-1}^1=C_{d-1}$ (if $\eta \neq 1$, note that $S(\eta)C_{d-1} = 0$). It is clear from the defining relations of $A_W^{\mathbb{C}}$ that the set $\left\{D_i^{\eta} ~|~i=1, \dots, e_{\eta}\right\}$ (with $e_{\eta}= d$ if $\eta = 1$ and $d-1$ otherwise) forms a basis of the $\eta$--eigenspace $E_s^{\eta}$ of $s$.

We first treat the case where $\eta=1$. The $1$--eigenspace $E_s^1$ of $s$ is given by 
$\mathbb{C} D_0^{1} \oplus  \dots \oplus \mathbb{C} D_{d-2}^{1} \oplus \mathbb{C} D_{d-1}^1,$
in particular we have $\dim (E_s^1)=d$. It is also clear from the relations in $A_W^{\mathbb{C}}$ that $C_1$ preserves this eigenspace, hence, that $E_s^1$ is an $A_W^{\mathbb{C}}$-submodule. We claim that, if $v+v^{-1}$ is not a root of the polynomial $Q_{d-1}(X,1)$, then $C_1$ has $d$ distinct eigenvalues on $E_s^1$, implying that $E_s^1$ is a direct sum of one-dimensional eigenspaces for $C_1$. This implies that $E_s^1$ is a direct sum of one-dimensional $A_W^{\mathbb{C}}$--submodules. Assume that 
$$C_1 \bigl(  a_0 D_0^1 + a_1 D_1^1 + \dots + a_{d-1} D_{d-1}^1\bigr)=\lambda \bigl(a_0 D_0^1 + a_1 D_1^1 + \dots + a_{d-1} D_{d-1}^1\bigr)$$ 
for some $\lambda, a_i\in\mathbb{C}$. The relations in $A_W^{\mathbb{C}}$ imply that 
\begin{multline*}
C_1 D_i^1= D_{i+1}^1 + D_{i-1}^1 \text{ if $1 \leq i \leq d-3$, } C_1 D_0^1=D_1^1, \\  C_1 D_{d-2}^1= d D_{d-1}^1 + D_{d-3}^1 \text{ and } C_1 D_{d-1}^1=\bigl(v+v^{-1}\bigr) D_{d-1}^1.
\end{multline*}
Hence the above equation can be rewritten as the following system
$$\begin{cases}
a_1= \lambda a_0\\
a_0+ a_2= \lambda a_1\\
\dots\\
a_{d-4}+a_{d-2}=\lambda a_{d-3}\\
a_{d-3}=\lambda a_{d-2}\\
d a_{d-2} + \bigl(v+v^{-1}\bigr) a_{d-1} = \lambda a_{d-1}.
\end{cases}$$   
The vector $(a_0, a_1, \dots, a_{d-1})$ is an eigenvector with eigenvalue $\lambda$ for the action of $C_1$ if and only if the above system of linear equations has infinitely many solutions, that is, if and only if the determinant of the matrix
$$M=\begin{bmatrix}
-\lambda & 1 & 0 & & & \cdots & 0\\
1 & -\lambda & 1 & 0 & & \cdots & 0\\
0 & 1 & -\lambda & 1 & 0 & \cdots & 0\\
\vdots& \cdots & & \ddots & & \cdots & \vdots\\
0 &\cdots & 0 & 1 & -\lambda & 1 & 0\\
0 &\cdots & &0 & 1 & -\lambda & 0\\
0 & \cdots & & &0  & d & -\lambda + \bigl(v+v^{-1}\bigr)
\end{bmatrix}$$
is equal to zero. For $1\leq i\leq d-1$, we denote by $M_{i}$ the matrix obtained by removing the last $d-i$ rows and columns of the matrix $M$. Observe that $\det(M)=\bigl(-\lambda+\bigl(v+v^{-1}\bigr) \bigr) \det(M_{d-1})$. Setting $R_0(\lambda)=1$ and $R_i(\lambda)=(-1)^i\det(M_i)$ and noticing that 
$$\det(M_i)= - \lambda \det( M_{i-1}) - \det( M_{i-2}),$$ 
we get $R_{i+1}(\lambda)= \lambda R_i(\lambda) - R_{i-1}(\lambda)$, while $R_1(\lambda)=\lambda$ and $R_0(\lambda)=1$. It follows that the polynomials $R_i$ satisfy the same inductive relation as the $Q_i(X,1)$. Hence $R_i(\lambda) = Q_i(\lambda,1)$ and it possesses $i$ distinct roots equal to $2\cos\bigl(\frac{k\pi}{i+1}\bigr)$ for $k=1, \dots, i$. This means that $\det(M) = 0$ if and only if $\lambda \in \left\{v+v^{-1}, 2\cos\bigl(\frac{k\pi}{d}\bigr) ~|~k=1, \dots, d-1 \right\}$. In particular, if, for all $k=1, \dots, d-1$, $v+v^{-1}\neq 2\cos\bigl(\frac{k\pi}{d}\bigr)$, then $C_1$ has $d$ distinct eigenvalues on $E_s^1$. This concludes in that case.

We now consider the case where $\eta$ is such that $\eta^d=1$, $\eta\neq 1$. The $\eta$-eigenspace $E_s^\eta$ of $s$ is then given by 
$\mathbb{C} D_0^{\eta} \oplus  \dots \oplus \mathbb{C} D_{d-2}^{\eta},$
in particular we have $\dim (E_s^\eta)=d-1$. It is also clear from the relations defining $A_W^{\mathbb{C}}$ that $C_1$ preserves $E_s^\eta$, hence that $E_s^\eta$ is an $A_W^{\mathbb{C}}$-submodule. We have
\begin{multline*}
C_1 D_i^{\eta}=D_{i+1}^{\eta}+ \eta D_{i-1}^{\eta} \text{ for all $i=1, \dots, d-3$,} \\
C_1 D_0^{\eta}=D_1^{\eta} \text{ and } C_1 D_{d-2}^{\eta}= \eta D_{d-3}^{\eta}.
\end{multline*}
Assume that
$$C_1 \bigl( a_0 D_0^{\eta} + a_1 D_1^{\eta} + \dots + a_{d-2} D_{d-2}^{\eta} \bigr)=\lambda \bigl(a_0 D_0^{\eta} + a_1 D_1^{\eta} + \dots + a_{d-2} D_{d-2}^{\eta} \bigr)$$
for some $\lambda, a_i\in\mathbb{C}$. This means that
$$\begin{cases}
\eta a_1= \lambda a_0\\
a_0+ \eta a_2= \lambda a_1\\
\dots\\
a_{d-4}+\eta a_{d-2}=\lambda a_{d-3}\\
a_{d-3}=\lambda a_{d-2}.\\
\end{cases}$$
As in the previous case, the vector $(a_0, a_1, \dots, a_{d-2})$ is an eigenvector with eigenvalue $\lambda$ for the action of $C_1$ if and only if the above system of linear equations has infinitely many solutions, that is, if and only if the determinant of the matrix
$$M=\begin{bmatrix}
-\lambda & \eta & 0 & & & \cdots & 0 \\
1 & -\lambda & \eta & 0 & & \cdots & 0\\
0 & 1 & -\lambda & \eta & 0 & \cdots & 0 \\
\vdots & \cdots & & \ddots & & \cdots & \vdots\\
0 & \cdots & 0 & 1 & -\lambda & \eta & 0\\
0 & \cdots & & 0 & 1 & -\lambda & \eta\\
0 & \cdots & & & 0 & 1 & -\lambda
\end{bmatrix}$$
is equal to zero. For $1\leq i\leq d-1$, we denote by $M_{i}$ the matrix obtained by removing the last $d-1-i$ rows and columns of the above matrix $M$. Note that here $M_{d-1}= M$.
Setting $R_0(\lambda)=1$ and $R_i(\lambda)=(-1)^i\det(M_i)$ and noticing that 
$$\det(M_i)= - \lambda \det( M_{i-1}) - \eta \det( M_{i-2}),$$ 
we get $R_{i+1}(\lambda)= \lambda R_i(\lambda) - \eta R_{i-1}(\lambda)$, while $R_1(\lambda)=\lambda$ and $R_0(\lambda)=1$. It follows that the polynomials $R_i$ satisfy the same inductive relation as the $Q_i(X,\eta)$. Hence $R_i(\lambda) = Q_i(\lambda,\eta)$ and it possesses $i$ distinct roots equal to $2\cos\bigl(\frac{k\pi}{i+1}\bigr)$ for $k=1, \dots, i$. This means that $\det(M) = \det(M_{d-1}) =0 $ if and only if $\lambda \in \left\{2\cos\bigl(\frac{k\pi}{d}\bigr) ~|~k=1, \dots, d-1 \right\}$ and hence that $C_1$ has $d-1$ distinct eigenvalues on $E_s^{\eta}$.

All the eigenspaces $\left\{E_s^{\eta}~|~\eta \in\C \text{ with }\eta^d=1\right\}$ are in direct sum and the sum of their dimensions is equal to $d+ (d-1)(d-1)=\dim A_W^{\C}$. Moreover, for $v$ generic, every such eigenspace splits as a direct sum of one-dimensional $C_1$--invariant subspaces (hence one-dimensional $A_W^{\C}$--submodules as $s$ and $C_1$ generate the algebra $A_W^{\C}$). We then get the following: 

\begin{thm}[Semisimplicity]\label{thm:semi}
Assume that $v+v^{-1}\neq 2\cos\bigl(\frac{k\pi}{d}\bigr)$ for all $k=1, \dots, d-1$. Then $A_W^{\C}$ is semisimple. 
\end{thm}


\begin{thebibliography}{4}

\bibitem{BoRo} C.~Bonnaf\'e and R. Rouquier, \textsl{An asymptotic cell category for cyclic groups}, https://arxiv.org/abs/math/1708.09730 (2017).

\bibitem{Bou} N.~Bourbaki, \textsl{Groupes et alg\`ebres de Lie, chapitres 4,5
et 6}, Hermann (1968).

\bibitem{BMR} M.~Brou\'e, G. Malle and R. Rouquier, \textsl{Complex reflection groups, braid groups, Hecke algebras}, J. Reine Angew. Math. {\bf 500} (1998), 127-190.

\bibitem{Spetses} M.~Brou\'e, G. Malle and J. Michel, \textsl{Towards spetses. I}, Transf. Groups {\bf 4} (1999), no. 2-3, 157-218. 

\bibitem{EW} B.~Elias and G.~Williamson, \textsl{The Hodge theory of Soergel
bimodules}, Ann.\ of Math.\ {\bf 180} (2014), 1089-1136.

\bibitem{EGST1} K.~Erdmann, E.L.~Green, N.~Snashall and R.~Taillefer, \textsl{Representation theory of the Drinfeld doubles of a family of Hopf algebras}, J. Pure Appl. Algebra {\bf 204} (2006), no. 2, 413-454. 

\bibitem{EGST2} K.~Erdmann, E.L.~Green, N.~Snashall and R.~Taillefer, \textsl{Stable Green ring of the Drinfeld doubles of the generalised Taft algebras (corrections and new results)}, https://arxiv.org/abs/1609.03831 (2016).

\bibitem{GT} T.~Gobet and A.-L. Thiel, \textsl{On generalized categories of Soergel bimodules in type $A_2$}, C. R. Ac. Sci. Paris, Ser. I {\bf 356} (2018), 258-263.

\bibitem{Humph} J.~Humphreys, \textsl{Reflection groups and Coxeter groups}, Cambridge Studies in
Advanced Mathematics {\bf 29}, Cambridge University Press (1990). 

\bibitem{Jensen} L.T.~Jensen, \textsl{The 2-braid group and Garside normal form}, Math. Z. {\bf 286} (2017), no. 1-2, 491-520.

\bibitem{KL} D.~Kazhdan and G.~Lusztig, \textsl{Representations of Coxeter
Groups and Hecke Algebras}, Invent.\ Math.\ {\bf 53} (1979), 165-184. 

\bibitem{KL2} D.~Kazhdan and G.~Lusztig, \textsl{Schubert varieties and Poincar\'e duality}, Geometry of the Laplace operator, Proc. Sympos. Pure Math. (1980), XXXVI, Amer. Math. Soc., Providence, R.I., 185-203.

\bibitem{KS} M. Khovanov and P. Seidel, \textsl{Quivers, Floer cohomology, and braid group actions},
J. Amer. Math. Soc. {\bf 15} (2002), no. 1, 203-271.

\bibitem{La1} A.~Lacabanne, \textsl{Slightly degenerate categories and $\Z$--modular data}, https://arxiv.org/abs/1807.00766 (2018).

\bibitem{La2} A.~Lacabanne, \textsl{Drinfeld double of quantum groups, tilting modules and $\Z$--modular data associated to complex reflection groups}, https://arxiv.org/abs/1807.00770 (2018).

\bibitem{Ma}  I.~Marin, \textsl{The freeness conjecture for Hecke algebras of complex reflection groups, and the case of the Hessian group $G_{26}$}, J. Pure Appl. Algebra {\bf 218} (2014), no. 4, 704-720.

\bibitem{Rouqprep} R.~Rouquier, \textsl{Categorification of braid groups}, http://arxiv.org/abs/math/0409593 (2004).

\bibitem{Rouq}  R.~Rouquier, \textsl{Categorification of $\mathfrak{sl}_2$ and braid groups}, Trends in representation theory of algebras and related topics, Contemp. Math. {\bf 406} (2006), Amer. Math. Soc., Providence, RI, 137-167.

\bibitem{Sperv} W. Soergel, \textsl{Kategorie $\mathcal{O}$, perverse Garben und Moduln über den Koinvarianten zur Weylgruppe}, J. Amer. Math. Soc. {\bf 3} (1990), no. 2, 421-445.

\bibitem{S1} W.~Soergel, \textsl{The combinatorics of Harish-Chandra bimodules}, J. Reine Angew. Math. {\bf 429} (1992), 49-74.

\bibitem{S} W.~Soergel, \textsl{Kazhdan-Lusztig polynomials and indecomposable bimodules over
polynomial rings}, J.\ Inst.\ Math.\ Jussieu {\bf 6} (2007), 501-525.


\end{thebibliography}
\end{document}